\documentclass{amsart}
\usepackage{amsmath,amsthm}
\usepackage{amsfonts,amssymb}
\usepackage{multirow}
\usepackage{enumerate}

\newtheorem{theorem}{Theorem}
\newtheorem{proposition}[theorem]{Proposition}
\newtheorem{lemma}[theorem]{Lemma}

\numberwithin{theorem}{section}

\theoremstyle{definition}

\newtheorem{ex}[theorem]{Example}

\newtheorem{corollary}{Corollary}
\theoremstyle{remark}
\newtheorem{remark}[theorem]{Remark}

 \def\b{{\beta}}

\allowdisplaybreaks[4]

\begin{document}

\title[The Dunkl kernel and intertwining operator for dihedral groups]{The Dunkl kernel and intertwining operator for dihedral groups}

\author{Hendrik De Bie}
\address{Department of Electronics and Information Systems \\Faculty of Engineering and Architecture\\Ghent University\\Krijgslaan 281, 9000 Gent\\ Belgium.}
\email{Hendrik.DeBie@UGent.be}

\author{Pan Lian}
\address{ School of Mathematical Sciences -- Tianjin Normal University\\
Binshui West Road 393, Tianjin 300387\\ P.R. China}
\email{pan.lian@outlook.com}


\begin{abstract}
Dunkl operators associated with finite reflection groups generate a commutative algebra of differential-difference operators. There exists a unique linear operator called intertwining operator which intertwines between this algebra and the algebra of standard differential operators. There also exists a generalization of the Fourier transform in this context called Dunkl transform.

In this paper, we determine an integral expression for the Dunkl kernel, which is the integral kernel of the Dunkl transform, for all dihedral groups. We also determine an integral expression for the intertwining operator in the case of dihedral groups, based on observations valid for all reflection groups.  As a special case, we recover the result of [Xu, Intertwining operators associated to dihedral groups. Constr. Approx. 2019]. Crucial in our approach is a systematic use of the link between both integral kernels and the simplex in a suitable high dimensional space.

\end{abstract}

\maketitle

\tableofcontents


\section{Introduction}
\setcounter{equation}{0}

Dunkl operators are differential-difference operators that generalize the standard partial derivatives. They are constructed using a finite reflection group and a parameter function on the orbits of this group on its root system. They were initially introduced by Charles Dunkl in \cite{DTAMS}, where he showed that, surprisingly, these operators still commute.
Dunkl operators have found a variety of applications in mathematics and mathematical physics.  They motivate the study of double affine Hecke algebras and Cherednik algebras. They are used in the study of probabilistic processes and are crucial for the integration of quantum many body problems of Calogero-Moser-Sutherland type. They have also made a lasting impact in the study of orthogonal polynomials and special functions in one and several variables.
Apart from the seminal book \cite{DX}, several excellent reviews are currently available on this topic. We refer the reader to e.g. \cite{DJap, rm} and to \cite{Anker} for a most recent state of the art.

The theory of Dunkl operators is further developed using two key ingredients. The first is the {\em Dunkl transform}, which is a generalization of the Fourier transform which now maps coordinate multiplication to the action of Dunkl operators and vice versa. It was introduced in \cite{Ddunkl} and further studied in \cite{deJ}, where the author showed the boundedness and analyticity of the associated integral kernel called Dunkl kernel, and obtained the Plancherel theorem. Various special cases of Paley-Wiener type theorems for the Dunkl transform were obtained in \cite{DJ} and multiplier theorems were investigated in e.g. \cite{Da,Dz}. The Dunkl transform was further generalized in \cite{SKO}, based on observations made in \cite{S2007}.

The second important operator is the {\em intertwining operator} $V$. This operator is a linear and homogeneous isomorphism on the space of polynomials that maps the standard partial derivatives to the Dunkl operators. It is abstractly proven that the intertwining operator can be represented as an integral operator \cite{rm1}. This is important if one wants to extend the intertwining property to larger function spaces. Also, for certain parameter values it behaves singularly, which were determined in \cite{DdO}.

From an abstract point of view, several satisfactory results are known about both the Dunkl transform and the intertwining operator, as mentioned above. However, it has remained a major problem for thirty years to develop explicit formulas of the Dunkl kernel and the intertwining operator for concrete choices of the finite reflection group. We describe and list the known results.
In the one dimensional case, both the Dunkl kernel and the intertwining operator are explicitly known, see \cite{DV, rm}. For the reflection groups of type $A_n$ the action of the intertwining operator on polynomials is determined in \cite{DIA} and given as a highly complicated integral for $n=2$ in \cite{DTAMS}. An integral expression for the intertwining operator for the group $B_2$ was obtained in \cite{D2}.
For dihedral groups, acting in the plane, the current state of the art is as follows. For polynomials, explicit formulas giving $V(z^k \bar z^\ell)$ as complicated sums of monomials are given in \cite{Ddih}, but no integral formula is provided. Using a Laplace transform technique and knowledge on the Poisson kernel obtained in \cite{D}, an explicit and concise expression was obtained for the Dunkl transform in the Laplace domain in \cite{CDL}.
 A series expression for the Dunkl kernel was given in \cite{DDY}. Recently, for a restricted class of functions and restrictions on the parameter function, an integral expression was given for the intertwining operator in \cite{x} using an integral over the simplex.

This lack of explicit formulas has seriously hindered the further development of harmonic analysis for Dunkl operators. The situation is better for the generalized Bessel function, which is a symmetrized version of the Dunkl kernel using the action of the reflection group.  For the root systems $A_{n}$, a complicated integral recurrence formula for the generalized Bessel function  was obtained in \cite{b1}. An explicit expression  is derived in \cite{Ps} for the intertwining operator in the symmetrized setting. For the dihedral case, a series expression for the generalized Bessel function  was  given in \cite{Dn} and  closed formulas were subsequently  obtained for some specific parameters there. An explicit expression in the Laplace domain for the generalized Bessel function was obtained in \cite{CDL}, using the same Laplace transform technique as for Dunkl dihedral kernel. Recently,  a Laplace type expression for the generalized Bessel function for even dihedral groups with one variable specified is given in \cite{DD2}.

The aim of the present paper is to give a complete  description in the case of dihedral groups of the Dunkl kernel, the generalized Bessel function and the intertwining operator. For the Dunkl kernel, we obtain an expression in terms of the  second class of Humbert functions (see Theorem \ref{m2}, \ref{ib111} )  or, alternatively, as an integral over the simplex (Theorem \ref{m2}). This is achieved by inverting the Laplace domain expression obtained in \cite{CDL}.
For the intertwining operator and its inverse, we first determine an integral expression which is new to our knowledge, linking it to the classical Fourier transform and the Dunkl kernel. This is achieved in Theorems \ref{gin} and \ref{iv2}, and the formula is valid for arbitrary reflection groups. Using the explicit formula for the Dunkl kernel in the dihedral case, we can subsequently give explicit expressions for the intertwining operator in that case. As expected, our formulas specialise to those given in \cite{x}.

The paper is organized as follows. In Section 2, we briefly introduce the basics of Dunkl theory and the second class of  Humbert functions. In Section 3,  we give a general integral expression for the intertwining operator and its inverse. Section 4 is devoted to the dihedral case.  We give the explicit formulas for the generalized Bessel function, Dunkl kernel and the intertwining operator. A new proof of Xu's result of \cite{x} can be found at the end of this section.
We end with conclusions and a list of notations used in Section 4.

\section{Preliminaries}
\subsection{Basics of Dunkl theory }
Let $G$ be a finite reflection group with a fixed positive root system $R_{+}$. A multiplicity function $\kappa: R \rightarrow \mathbb{C}$ on the root system $R$ is a $G$-invariant function, i.e. $\kappa(\alpha)=\kappa(h\cdot\alpha)$ for all $h\in G$.
For $\xi \in \mathbb{R}^{m}$, the Dunkl operator $T_{\xi}$ on $\mathbb{R}^{m}$ associated with the group $G$ and the multiplicity function $\kappa(\alpha)$ is defined by
\[T_{\xi}(\kappa)f(x)=\partial_{\xi}f(x)+\sum_{\alpha\in R_{+}}\kappa(\alpha)\langle \alpha,\xi\rangle\frac{f(x)-f(\sigma_{\alpha}x)}{\langle \alpha, x\rangle}, \qquad x\in \mathbb{R}^{m}\]
 where $\langle \cdot, \cdot\rangle$  is the canonical Euclidean  inner product in $\mathbb{R}^{m}$ and $\sigma_{\alpha}x:=x-2\langle x, \alpha\rangle\alpha/||\alpha||^{2}$ is a reflection.
 In the sequel, we write $T_{j}$ in place of $T_{e_{j}}(\kappa)$ where $e_{j}, j=1,\cdots, m$ is a vector of the standard basis of $\mathbb{R}^{m}$. The Dunkl Laplacian  $\Delta_{\kappa}$ is then defined by $\Delta_{\kappa}=\sum_{j=1}^{m}T_{j}^{2}$.

 Let $\mathcal{P}$ denote the polynomials on $\mathbb{R}^{m}$ and $\mathcal{P}_{n}$ the homogeneous polynomials of degree $n$. The Dunkl operators $\{T_{j}\}$ generate a commutative algebra of differential-difference operators on $P$. Each $T_{i}$ is homogeneous of degree $-1$. For $p, q\in \mathcal{P}$, the Fischer bilinear form
 \[ [p,q]_{\kappa}:=(p(T)q)(0)\]
  was introduced in \cite{DX}. Here $p(T)$  is the differential-difference operator obtained by replacing $x_{j}$ in $p$ by $T_{j}$. The Macdonald identity \cite{M2} also has a useful generalization in the Dunkl setting as follows:
   \[ [p,q]_{\kappa}=c_{\kappa}^{-1}\int_{\mathbb{R}^{m}}\left(e^{-\Delta_{\kappa}/2}p(x)\right)\left(e^{-\Delta_{\kappa}/2}q(x)\right)e^{-|x|^{2}/2}\omega_{\kappa}(x)dx \]
where the weight function is \[\omega_{\kappa}(x)=\prod_{\alpha\in R_{+}}|\langle \alpha, x\rangle|^{2\kappa(\alpha)}\] and  $c_{\kappa}$ is the Macdonald-Mehta-Selberg constant, i.e.  \[c_{\kappa}=\int_{\mathbb{R}^{m}}e^{-|x|^{2}/2}\omega_{\kappa}(x)dx.\]

The Dunkl kernel $E_{\kappa}(x, y)$ is the joint eigenfunction of all the $T_{j}$,
\begin{eqnarray}\label{p1}T_{j}E_{\kappa}(x, y)=y_{j}E_{\kappa}(x, y), \qquad j=1,\ldots, m \end{eqnarray}
and satisfies $E_{\kappa}(0, y)=1$. When $\kappa=0$, it reduces to the ordinary exponential function $e^{\langle x, y\rangle}$.
Choosing an orthonomal basis $\{\varphi_{\nu},\nu\in \mathbb{Z}_{+}^{N}\}$ of $\mathcal{P}$ with respect to the Fischer inner product, the Dunkl kernel  can be expressed as
 \begin{eqnarray}\label{ks1} E_{\kappa}(x,y)=\sum_{\nu\in \mathbb{Z}_{+}^{N}}\varphi_{\nu}(x)\varphi_{\nu}(y)\end{eqnarray}
for all $x, y \in \mathbb{R}^{m}$. We further introduce the generalized Bessel function  defined as the symmetric version of the Dunkl kernel  by
\[\mathcal{J}_{\kappa}(x, y)=\frac{1}{|G|}\sum_{g\in G}E_{\kappa}(x, g\cdot y).\]

 The Dunkl transform is defined using the joint eigenfunction  $E_{\kappa}(x, y)$ and the weight function $\omega_{\kappa}(x)$
by
\[\mathcal{F}_{\kappa}f(y):=c_{\kappa}^{-1}\int_{\mathbb{R}^{m}}E_{\kappa}(-ix, y)f(x)\omega_{\kappa}(x)dx \quad (y\in \mathbb{R}^{m}). \]
When $\kappa=0$, the Dunkl transform reduces to the ordinary Fourier transform $\mathcal{F}$, i.e.
\[\mathcal{F}f(y):=\frac{1}{(2\pi)^{m/2}}\int_{\mathbb{R}^{m}}e^{-i\langle x, y\rangle}f(x)dx.  \]
The definition of Dunkl transform is motivated by the following proposition (Proposition 7.7.2 in \cite{DX})  which will also be used in the following section.
\begin{proposition}\label{h2}  Let $p$ be a polynomial on $\mathbb{R}^{m}$ and $v(y)=\sum_{j=1}^{m}y_{j}^{2}$ for $y\in \mathbb{C}^{m}$, then
\[c_{\kappa}^{-1}\int_{\mathbb{R}^{m}}  \left[e^{-\Delta_{\kappa}/2}p(x)\right]E_{\kappa}(x, y)e^{-|x|^{2}/2}\omega_{\kappa}(x)dx=e^{v(y)/2}p(y).  \]
\end{proposition}

There exists an unique linear and homogenous isomorphism on $\mathcal{P}$ which satisfies $V_{\kappa}1=1$  and intertwines  the ordinary partial differential operators and the Dunkl operators,
\begin{eqnarray} \label{it1}T_{j}V_{\kappa}=V_{\kappa}\partial_{j}, \qquad j=1, 2,\ldots, m.\end{eqnarray}
The operator $V_{\kappa}$ is called the Dunkl intertwining operator in the literature. The explicit representation of this operator is only known so far for some special cases, e.g. the group $\mathbb{Z}_{2}$ and root system $ A_{2}, B_{2}$, see \cite{D1, DX} and \cite{Anker} for a recent review. We  list the rank one case, which is frequently used in this paper.
\begin{ex} For $k>0$, the intertwining operator in the rank one case is
\begin{eqnarray} \label{kk1}
 V_{k}(p)(x)=\frac{\Gamma(k+1/2)}{\Gamma(1/2)\Gamma(k)}\int_{-1}^{1}p(xt)(1-t)^{k-1}(1+t)^{k}dt
 .\end{eqnarray}
\end{ex}

 Using the intertwining operator, the Dunkl kernel is expressed as
\[E_{\kappa}(x, y)=V_{\kappa}\left(e^{\langle \cdot, y\rangle}\right)(x).\]
Throughout this paper, we only consider real multiplicity functions with $\kappa\ge 0$.
\subsection{Humbert function of several variables}
There are many ways to define  hypergeometric functions of several variables. In this subsection, we  introduce the  second  Humbert function of $m$ variables
\[\Phi_{2}^{(m)}(\beta_{1},
\ldots, \beta_{m}; \gamma; x_{1}, \ldots, x_{m}):=\sum_{j_{1},\ldots, j_{m}\ge 0}\frac{(\beta_{1})_{j_{1}}\cdots(\beta_{m})_{j_{m}}}{(\gamma)_{j_{1}+\cdots+j_{m}}}\frac{x_{1}^{j_{1}}}{j_{1}!}\cdots \frac{x_{m}^{j_{m}}}{j_{m}!}  \]
which is one of  the confluent Lauricella hypergeometric series, see \cite{EH} (Chapter 2, 2.1.1.2, page 42).

When  $\gamma-\sum_{j=1}^{m}\beta_{j}$  and each $\beta_{j}$,  $j=1, 2, \ldots, m$ is a positive number, the hypergeometric function $\Phi_{2}^{(m)}$ admits the following integral representation
\begin{eqnarray}\label{pi1}&&\Phi_{2}^{(m)}(\beta_{1},
\ldots, \beta_{m}; \gamma; x_{1}, \ldots, x_{m})\\&=&C_{\beta}^{ (\gamma)}
\int_{T^{m}} e^{\sum_{j=1}^{m}x_{j}t_{j} } \left(1-\sum_{j=1}^{m}t_{j}\right)^{\gamma-\sum_{j=1}^{m}\beta_{j}-1}\prod_{j=1}^{m}t_{j}^{\beta_{j}-1}dt_{1}\ldots dt_{m}\nonumber
\end{eqnarray}
where \[C_{\beta}^{ (\gamma)}=\frac{\Gamma{(\gamma)}}{\Gamma(\gamma-\sum_{j=1}^{m}\beta_{j})\prod_{j=1}^{m}\Gamma(\beta_{j})}\] and  $T^{m}$ is the open unit simplex in $\mathbb{R}^{m}$ given by
\begin{eqnarray*}T^{m}=\left\{(t_{1},\ldots, t_{m}): t_{j}>0, j=1,\ldots, m, \sum_{j=1}^{m}t_{j}<1\right\}.\end{eqnarray*}
See \cite{cw, H} for more details on this integral expression.

Moreover,   Section 4.24, formula (5) in \cite{E2} shows that the Laplace transform in the variable $t$ of  $\Phi_{2}^{(m)}$  is given by
\begin{eqnarray}\label{l1} &&\mathcal{L}(t^{\gamma-1}\Phi_{2}^{(m)}(\beta_{1},\ldots,\beta_{m}; \gamma;\lambda_{1}t,\ldots,\lambda_{m}t) )\\&=&\frac{\Gamma(\gamma)}{s^{\gamma}}\left(1-\frac{\lambda_{1}}{s}\right)^{-\beta_{1}}\cdots\left(1-\frac{\lambda_{m}}{s}\right)^{-\beta_{m}} \nonumber
\end{eqnarray}
with ${\rm Re}\, \gamma, {\rm Re}\,s>0,  {\rm Re}\lambda_{j}, j=1,\cdots, m .$
\section{New formulas for the intertwining operator and its inverse
}\label{se1}
In this section, we give an integral expression of $V_{\kappa}$ in terms of the Dunkl kernel which will be used to derive the explicit expression for the dihedral groups in subsequent sections.
We first formulate the following  lemma using the  ordinary Fourier transform.  Note that the notation $e^{-\Delta_{(y)}/2} p(iy)$  used in the following means acting with the operator $e^{-\Delta_{(y)}/2}$ on the complex valued polynomial $p(iy)$.
\begin{lemma}  For any $y, z \in \mathbb{R}^{m}$, let $p(z)$ be a polynomial. Then we have
 \begin{eqnarray} \label{ep1}
 e^{-\Delta_{(y)}/2}p(-iy)&=&\frac{1}{(2\pi)^{m/2}}\int_{\mathbb{R}^{m}} e^{-i\langle x, y\rangle}e^{-|x|^{2}/2}p(x)e^{|y|^{2}/2}dx, \nonumber\\
 e^{-\Delta_{(y)}/2}E_{\kappa}(iy,z)&=&\frac{1}{(2\pi)^{m/2}}\int_{\mathbb{R}^{m}} e^{-i\langle x, y\rangle}e^{-|x|^{2}/2}E_{\kappa}(-x,z)e^{|y|^{2}/2}dx,
\end{eqnarray}
where $\Delta_{(y)}$ is the usual Laplace operator $\Delta=\sum_{j=1}^{m}\partial_{j}^{2}$ acting on the variable $y$.
\end{lemma}
\begin{proof} Let $\kappa=0$ in Proposition \ref{h2}, then one has
\[\int_{\mathbb{R}^{m}}\biggl(e^{-\Delta_{(y)}/2}p(-iy)\biggr)e^{i\langle x, y\rangle} e^{-|y|^{2}/2}dy=(2\pi)^{m/2}e^{-|x|^{2}/2}p(x).\]
Acting with the ordinary Fourier transform $\mathcal{F}$ on both sides leads to the first identity in the present theorem,
\[\biggl(e^{-\Delta_{(y)}/2}p(-iy)\biggr)e^{-|y|^{2}/2}=\frac{1}{(2\pi)^{m/2}}\int_{\mathbb{R}^{m}} e^{-i\langle x, y\rangle}e^{-|x|^{2}/2}p(x)dx.  \]

For the second part, we first give the following estimation,
\begin{eqnarray} \label{h1} \left|e^{-\Delta_{(y)}/2}E_{\kappa}(iy,z)\right|&=&\left|V_{\kappa}\left(e^{-\Delta_{(y)}/2}e^{i\langle y, \cdot\rangle}\right)(z)\right|\nonumber\\&=&\left|V_{\kappa}\left(e^{|\cdot|^{2}/2}e^{i\langle y, \cdot\rangle}\right)(z)\right|\nonumber\\&\le& e^{|z|^{2}/2}
\end{eqnarray}
where the last inequality is by $\left|e^{|x|^{2}/2}e^{i\langle y, x\rangle}\right|\le e^{|x|^{2}/2}$ and the Bochner-type representation of the intertwining operator  proved in \cite{rm1}. Using the above estimation (\ref{h1}) and the dominated convergence theorem, the second equality in the theorem is obtained by replacing $p$ with the Dunkl kernel in the first equality.
\end{proof}
By the Fischer inner product, we  give a general formula for the intertwining operator,
which reveals the relationship between the Dunkl kernel and the intertwining operator. In principle, once the Dunkl kernel is known, our theorem yields the integral expression for the intertwining operator.
\begin{theorem}\label{gin} Let $p$ be a polynomial and  $K(iy,z):=e^{-\Delta_{(y)}/2}E_{\kappa}(iy,z)$,
   then for any $z\in \mathbb{R}^{m}$, the intertwining operator $V_{\kappa}$ satisfies
   \begin{eqnarray}\label{int1}V_{\kappa}(p)(z)&=&\frac{1}{(2\pi)^{m/2}} \int_{\mathbb{R}^{m}}K(iy, z)\mathcal{F}\left(p(\cdot)e^{-|\cdot|^{2}/2}\right)(y) dy \nonumber\\
   &=&\frac{1}{(2\pi)^{m/2}}\int_{\mathbb{R}^{m} } \mathcal{F}\left(e^{-|\cdot|^{2}/2}E_{\kappa}(\cdot, -z)\right)(y)\mathcal{F}\left(e^{-|\cdot|^{2}/2} p(\cdot)\right)(y)e^{|y|^{2}/2}dy.
   \end{eqnarray}
\end{theorem}
\begin{proof} It is well known that the exponential function $e^{\langle y, z\rangle }$ is the reproducing kernel of the polynomials with respect to the classical Fischer inner product,  denoted by  $[\cdot , \cdot]_{0}$, corresponding to the multiplicity function $\kappa(\alpha)=0$.
More precisely, with the  Macdonald identity,  we have
\begin{eqnarray*} p(z)&=&\left[p(y),e^{\langle y, z\rangle}\right]_{0}\\&=&\frac{1}{(2\pi)^{m/2}}\int_{\mathbb{R}^{m}}\biggl(e^{-\Delta_{(y)}/2}p(y)\biggr)\biggl(e^{-\Delta_{(y)}/2}e^{\langle y, z\rangle}\biggr)e^{-|y|^{2}/2}dy. \end{eqnarray*}
By complexification, we have
\begin{eqnarray} \label{ks3}p(z)&=&\left[p(-iy), e^{i\langle y, z\rangle}\right]_{0}\nonumber\\&=&\frac{1}{(2\pi)^{m/2}}\int_{\mathbb{R}^{m}}\biggl(e^{-\Delta_{(y)}/2}p(-iy)\biggr)\biggl(e^{-\Delta_{(y)}/2}e^{i\langle y, z\rangle}\biggr)e^{-|y|^{2}/2}dy. \end{eqnarray}
Applying $V_{\kappa}$ on both sides of (\ref{ks3}) with respect to $z$, it follows that
 \begin{eqnarray} \label{ks4}&&V_{\kappa}(p)(z)\nonumber\\&=&\frac{1}{(2\pi)^{m/2}}\int_{\mathbb{R}^{m}}\biggl(e^{-\Delta_{(y)}/2}p(-iy)\biggr)V_{\kappa}\biggl(e^{-\Delta_{(y)}/2}e^{i\langle y, \cdot\rangle}\biggr)(z)e^{-|y|^{2}/2}dy\nonumber\\&=&\frac{1}{(2\pi)^{m/2}}\int_{\mathbb{R}^{m}}\biggl(e^{-\Delta_{(y)}/2}p(-iy)\biggr)\biggl(e^{-\Delta_{(y)}/2}E_{\kappa}(iy,z)\biggr)e^{-|y|^{2}/2}dy .\end{eqnarray}
 Now, putting the integral expression (\ref{ep1}) in (\ref{ks4}),  leads to
 \begin{eqnarray*} V_{\kappa}(p)(z)=\frac{1}{(2\pi)^{m/2}} \int_{\mathbb{R}^{m}}K(iy, z)\mathcal{F}\left(e^{-|x|^{2}/2}p(x)\right) dy.
 \end{eqnarray*}
The second identity in (\ref{int1}) follows from the expression (\ref{ep1}).
\end{proof}

\begin{remark} It is known that for any polynomial $p(x)$, the function $ \mathcal{F}\left(e^{-|x|^{2}/2}p(x)\right)$  is in the Schwarz space. Furthermore, by the  estimation
\[|K(iy, z)|=\left|V_{\kappa}\left(e^{-\Delta_{(y)}/2}e^{i\langle y, \cdot\rangle}\right)(z)\right|\le e^{|z|^{2}/2},\]
the integral in (\ref{int1}) makes sense.
\end{remark}
\begin{remark} At the moment, it is not clear to the authors if the above expression (\ref{gin}) can be rewritten  as a R\"osler type integral \cite{rm1}
\[V_{\kappa}(p)(z)=\int_{\mathbb{R}^{m}}p(x)d\mu_{z}(x).  \]
\end{remark}

For a polynomial which is invariant under $G$, i.e. $p(g\cdot z)=p(z),$ for all $g
\in G$, the reproducing kernel under the ordinary Fischer inner product is given by \[\frac{1}{|G|}\sum_{g\in G} e^{\langle x,g\cdot y\rangle}.\]  Therefore,  the intertwining operator acting on  invariant polynomials could be obtained as
\[ V_{\kappa}(p)(z)=V_{\kappa}\left(\left[p(\cdot),\frac{1}{|G|}\sum_{g\in G} e^{\langle \cdot,g\cdot y\rangle} \right]_{0}\right)(z).\]
Following the proof of Theorem \ref{gin}, we obtain the expression of the intertwining operator on the invariant polynomials by replacing the Dunkl kernel with  the generalized Bessel function.
\begin{corollary}\label{cor} Let $p$ be a $G$-invariant polynomial.  Then for any $z\in \mathbb{R}^{m}$, the intertwining operator $V_{\kappa}$ satisfies
   \begin{eqnarray*}V_{\kappa}(p)(z)
   &=&\frac{1}{(2\pi)^{m/2}}\int_{\mathbb{R}^{m} } \mathcal{F}\left(e^{-|\cdot|^{2}/2}\mathcal{J}_{\kappa}(\cdot, -z)\right)(y)\mathcal{F}\left(e^{-|\cdot|^{2}/2} p(\cdot)\right)(y)e^{|y|^{2}/2}dy.
   \end{eqnarray*}

\end{corollary}
In the following, we show independently that the integral expression given in (\ref{int1}) actually intertwines the usual partial derivatives and the Dunkl operator, i.e. the relations (\ref{it1}). We use the notation $[A, B]:=AB-BA$ for the commutator of  two operators $A, B$. The following lemma from \cite{rm} will help to verify these relations.
\begin{lemma}\cite{rm} \label{rle}
For $j=1,\ldots, m$, we have
\begin{eqnarray*}
&&[y_{j}, \Delta_{\kappa}/2]=-T_{j};\\
&&\left[y_{j}, e^{-\Delta_{\kappa}/2}\right]=T_{j}e^{-\Delta_{\kappa}/2}.
\end{eqnarray*}
\end{lemma}
\begin{theorem} For any polynomial $p$, the integral operator \begin{eqnarray*}
V_{\kappa}(p)(z)=\frac{1}{(2\pi)^{m/2}} \int_{\mathbb{R}^{m}} K(iy, z) \mathcal{F}\left(p(\cdot)e^{-|\cdot|^{2}/2}\right)(y)dy
\end{eqnarray*} satisfies the relations \[ V_{\kappa} \partial_{j}= T_{j} V_{\kappa},\qquad j=1,\ldots, m.\]\end{theorem}
\begin{proof}
By direct computation, we have
\begin{eqnarray*}
&&V_{\kappa}(\partial_{j}p)(z)\\&=&\frac{1}{(2\pi)^{m/2}} \int_{\mathbb{R}^{m}} K(iy, z) \mathcal{F}\left( (\partial_{j}p)(\cdot)e^{-|\cdot|^{2}/2}\right)(y)dy\\
&=&\frac{1}{(2\pi)^{m/2}}\int_{\mathbb{R}^{m}}K(iy,z) \mathcal{F}\left( \partial_{j}(p(x)e^{-|x|^{2}/2})+x_{j}p(x)e^{-|x|^{2}/2} \right)(y)dy\\
&=&\frac{1}{(2\pi)^{m/2}}\int_{\mathbb{R}^{m}}K(iy,z) i(y_{j}+\partial_{y_{j}}) \mathcal{F}\left(p(x)e^{-|x|^{2}/2}\right)(y)dy\\
&=&\frac{1}{(2\pi)^{m/2}}\int_{\mathbb{R}^{m}} \left(i(y_{j}-\partial_{y_{j}})K(iy,z)\right) \mathcal{F}\left(p(x)e^{-|x|^{2}/2}\right)(y)dy\\
&=&\frac{1}{(2\pi)^{m/2}}\int_{\mathbb{R}^{m}} \left(e^{-\Delta_{(y)}/2}\left(iy_{j}E_{\kappa}(iy,z)\right)\right) \mathcal{F}\left(p(x)e^{-|x|^{2}/2}\right)(y)dy\\
&=&T_{j}V_{\kappa}(p)(z)
\end{eqnarray*}
where we have used Lemma \ref{rle} in the fifth equality and the relation (\ref{p1}) in the last step.
\end{proof}

In the following, we compute some special cases using the  integral expression (\ref{int1}). The explicit expressions for the intertwining operator  associated to  some special root systems are reobtained.
\begin{ex} When $\kappa_{\alpha}=0$,  expression (\ref{int1}) reduces to
 \begin{eqnarray*}V_{\kappa}(p)(z)&=&\frac{1}{(2\pi)^{m/2}} \int_{\mathbb{R}^{m}} K(iy, z) \mathcal{F}\left(p(x)e^{-|x|^{2}/2}\right)(y)dy\\&=&
\frac{1}{(2\pi)^{m/2}} \int_{\mathbb{R}^{m}} e^{i\langle y, z \rangle }e^{|z|^{2}/2} \mathcal{F}\left(p(x)e^{-|x|^{2}/2}\right)(y)dy \\&=&e^{|z|^{2}/2}\mathcal{F}^{-1}\left(\mathcal{F}\left(p(x)e^{-|x|^{2}/2}\right)\right)(z)\\&=&p(z)
\end{eqnarray*}
This coincides with the result that for $\kappa_{\alpha}=0$ the intertwining operator is the identity operator.
\end{ex}
\begin{ex} (Rank 1 case) For the rank one case with $k>0$, the Dunkl kernel is explicitly known as
\[E_{k}(iy, z)=\frac{\Gamma(k+1/2)}{\Gamma(1/2)\Gamma(k)}\int_{-1}^{1}e^{ityz}(1-t)^{k-1}(1+t)^{k}dt.  \]
It follows that
 \begin{eqnarray*}e^{-\Delta_{(y)}/2}E_{k}(iy,z)=\frac{\Gamma(k+1/2)}{\Gamma(1/2)\Gamma(k)}\int_{-1}^{1}e^{ityz}e^{t^{2}z^{2}/2}(1-t)^{k-1}(1+t)^{k}dt.
\end{eqnarray*}
Now, we have
 \begin{eqnarray*} V_{k}(p)(z)
 &=&\frac{\Gamma(k+1/2)}{(2\pi)^{1/2}\Gamma(1/2)\Gamma(k)}\int_{\mathbb{R}}\left(e^{-\Delta_{(y)}/2}p(-iy)\right) e^{-|y|^{2}/2}\\&&\times\int_{-1}^{1}e^{ityz}e^{t^{2}z^{2}/2}(1-t)^{k-1}(1+t)^{k}dtdy\\
  &=&\frac{\Gamma(k+1/2)}{(2\pi)^{1/2}\Gamma(1/2)\Gamma(k)}\int_{-1}^{1}\left(\int_{\mathbb{R}}\left(e^{-\Delta_{(y)}/2}p(-iy) \right)  e^{-|y|^{2}/2}e^{ityz}dy \right)\\&&\times e^{t^{2}z^{2}/2}(1-t)^{k-1}(1+t)^{k}dt\\
 &=&\frac{\Gamma(k+1/2)}{\Gamma(1/2)\Gamma(k)}\int_{-1}^{1}p(tz)(1-t)^{k-1}(1+t)^{k}dt
 \end{eqnarray*}
 where we have used Proposition \ref{h2}  again in the last equality.
\end{ex}
\begin{ex}(Root system $B_{2}$) When $\gamma=\kappa_{1}+\kappa_{2}>1/2$, the generalized Bessel function of type $B_{2}$ (i.e. the dihedral group $I_{4}$) admits the following Laplace-type integral representation:
\[\mathcal{J}_{\kappa}(iy, z)=\int_{\mathbb{R}^{2}}e^{i\langle y, x\rangle}H_{k}(x, z) dz \]
where $H_{k}(x, z)$ is a positive function with explicit expression, see \cite{AD}. It follows that
 \begin{eqnarray*}e^{-\Delta_{(y)}/2}\mathcal{J}_{\kappa}(iy,z)=\int_{\mathbb{R}^{2}}e^{i\langle y, x\rangle}e^{|x|^{2}/2}H_{k}(x, z) dz.
\end{eqnarray*}
Now, for the root system $B_{2}$ and $p(z)$ a $I_{4}$-invariant polynomial, which means that $p(g\cdot z)=p(z)$ for $ g\in I_{4}$, we have
\begin{eqnarray*} &&V_{\kappa}(p)(z)\\&=&\frac{1}{2\pi}\int_{\mathbb{R}^{2}}\left(e^{-\Delta_{(y)}/2}p(-iy)\right)\left(e^{-\Delta_{(y)}/2}\mathcal{J}_{\kappa}(iy,z)\right)e^{-|y|^{2}/2}dy \\&=& \frac{1}{2\pi} \int_{\mathbb{R}^{2}}\left(e^{-\Delta_{(y)}/2}p(-iy)\right)\left(\int_{\mathbb{R}^{2}}e^{-\Delta_{(y)}/2}e^{i\langle y, x\rangle}H_{k}(x, z) dz\right)e^{-|y|^{2}/2}dy
\\&=& \frac{1}{2\pi} \int_{\mathbb{R}^{2}}\left[\int_{\mathbb{R}^{2}}\left(e^{-\Delta_{(y)}/2}p(-iy)\right)\left(e^{-\Delta_{(y)}/2}e^{i\langle y, x\rangle}\right)e^{-|y|^{2}/2}dy\right] H_{k}(x, z) dz
\\&=&\int_{\mathbb{R}^{2}}p(x)H_{k}(x, z) dz.
\end{eqnarray*}
In general, whenever a  Laplace-type representation of the generalized Bessel function is obtained, it is  possible to obtain the intertwining operator by the method described  in this example.
\end{ex}

While the precise structure of the  Dunkl intertwining operator has been unknown for a long time, the formal inverse of this operator is easy to get, see \cite{DX}. This operator can be expressed as \[V_{\kappa}^{-1}(p)(x)={\rm exp}\left(\sum_{j=1}^{m}x_{j}T_{j}^{(y)}\right)p(y)\bigg|_{y=0}= \left[e^{\langle x, y\rangle},p(y)\right]_{\kappa} \]
and satisfies the following relation \[V_{\kappa}^{-1}T_{j}=\partial_{j}V_{\kappa}^{-1}, \qquad j=1, 2,\ldots, m.\]
The integral expression of $V_{\kappa}^{-1}$ can be obtained similarly as in Theorem \ref{gin}. We omit the proof here.
\begin{theorem} \label{iv2} Let $p(z)$ be a polynomial. Denote  \[L(iy,z):=e^{-\Delta_{\kappa,y}/2}e^{i\langle y,z\rangle}=c_{\kappa}^{-1}\int_{\mathbb{R}^{m}}E_{\kappa}(-ix,y) e^{-\langle x, z\rangle}e^{-|x|^{2}/2}e^{|y|^{2}/2}\omega_{\kappa}(x)dx.\]
 Then the inverse of the intertwining operator $V_{\kappa}^{-1}$ satisfies \begin{eqnarray*}V_{\kappa}^{-1}(p)(z)= c_{\kappa}^{-1} \int_{\mathbb{R}^{m}} L(iy,z)\mathcal{F}_{\kappa}\left(p(\cdot)e^{-|\cdot|^{2}/2}\right)(y) \omega_{\kappa}(y)dy\end{eqnarray*}
 where $\mathcal{F}_{\kappa}$ is the Dunkl transform.
\end{theorem}
\begin{theorem} For any polynomial $p(x)$, the integral operator $V_{\kappa}^{-1}$ satisfies \[V_{\kappa}^{-1}T_{j}=\partial_{j}V_{\kappa}^{-1},\quad j=1,\ldots, m.\]\end{theorem}
\begin{proof} Since $e^{-|x|^{2}/2}$ is $G$-invariant, we have by direct computation,
\begin{eqnarray}\label{iv1}T_{j}(p(x)e^{-|x|^{2}/2})=T_{j}(p) e^{-|x|^{2}/2}+p(x) T_{j}(e^{-|x|^{2}/2}).\end{eqnarray}
Now, the intertwining relations are verified as follows
\begin{eqnarray*}
&&V_{\kappa}^{-1}(T_{j}p)(z)\\&=&c_{\kappa}^{-1}  \int_{\mathbb{R}^{m}} L(iy, z) \mathcal{F}_{\kappa}\left( (T_{j}p)(\cdot)e^{-|\cdot|^{2}/2}\right)(y)\omega_{\kappa}(y)dy\\
&=&c_{\kappa}^{-1} \int_{\mathbb{R}^{m}}L(iy,z) \mathcal{F}_{\kappa}\left( T_{j}(p(x)e^{-|x|^{2}/2})+x_{j}p(x)e^{-|x|^{2}/2} \right)(y)\omega_{\kappa}(y)dy\\
&=&c_{\kappa}^{-1} \int_{\mathbb{R}^{m}}L(iy,z) i(y_{j}+T_{y_{j}}) \mathcal{F}_{\kappa}\left(p(x)e^{-|x|^{2}/2}\right)(y)\omega_{\kappa}(y)dy\\
&=&c_{\kappa}^{-1} \int_{\mathbb{R}^{m}} \left[i(y_{j}-T_{y_{j}})L(iy,z)\right] \mathcal{F}_{\kappa}\left(p(x)e^{-|x|^{2}/2}\right)(y)\omega_{\kappa}(y)dy\\
&=&c_{\kappa}^{-1} \int_{\mathbb{R}^{m}} \left[e^{-\Delta_{\kappa, y}/2}\left(iy_{j}e^{\langle iy,z\rangle}\right)\right] \mathcal{F}_{\kappa}\left(p(x)e^{-|x|^{2}/2}\right)(y)\omega_{\kappa}(y)dy\\
&=&\partial_{z_{j}}V_{\kappa}^{-1}(p)(z),
\end{eqnarray*}
where we have used formula (\ref{iv1}) in the second line and Lemma \ref{rle} for the fifth equality.
\end{proof}

\section{The case of dihedral groups}
 The dihedral group $I_{k}$ is the group of symmetries of the regular $k-$gon. We use complex coordinates $z=x_{1}+ix_{2}$ and identify $\mathbb{R}^{2}$ with $\mathbb{C}$. Each reflection  $\sigma_{j}$ in $I_{k}$ is given by $z\rightarrow \overline{z}e^{ij2\pi/k}$, $0\le j\le 2k-1.$ In this section, we will consider the explicit expressions for the Dunkl kernel and the intertwining operator associated to the dihedral groups. These can be achieved mainly due to the reduction method, i.e.  the general $I_{k}$ case  can be reduced to the explicitly known cases $\mathbb{Z}_{2}$ and $\mathbb{Z}_{2}^{2}$, see  Section 7.6 in \cite{DX}. In turn, the reduction determines the representations of the formulas. As we will see, the Dunkl kernel and the generalized Bessel function are expressed as the compositions of two integrals.  One is the Weyl fractional integral ( the intertwining operator associated to the group $\mathbb{Z}_{2}$) and the other one corresponds to the reduction.

   On the other hand, the dihedral group $I_{k}$ is generated by the rotation $ z\rightarrow ze^{i2\pi/k}$ and the reflection $z\rightarrow \overline{z}$. 
   The action of the rotations and reflection can be seen from the formula, see  Theorem \ref{bess1} and Theorem \ref{ib111}. The positivity and bounds of the Dunkl kernel and generalized Bessel function will be seen directly from the formulas as well.

\subsection{The generalized Bessel function}
It was proved  in \cite{Dn} that the generalized Bessel function can be expressed as a  symmetric beta integral of an infinite series.  Let $z=|z|e^{i\phi_{1}}$, $w=|w|e^{i\phi_{2}}$ and $\xi_{u,v}(\phi_{1}, \phi_{2})=v\cos (\phi_{1})\cos (\phi_{2})+u\sin(\phi_{1})\sin (\phi_{2})$, then we have the following expression.  

\begin{theorem}\cite{Dn} For the dihedral group $I_{2k}$, $k\ge 2$ and non-negative multiplicity function $\kappa=(\alpha, \beta)$,  the generalized Bessel function  is given by
\begin{eqnarray}\label{B1}\mathcal{J}_{\kappa}(z, w)
&=&\frac{1}{2}\int_{-1}^{1}\int_{-1}^{1}\biggl(f_{2k, \alpha+\beta}(|zw|,\xi_{u,v}(k\phi_{1}, k\phi_{2}), 1)\\&&+f_{2k,\alpha+\beta}(|zw|,-\xi_{u,v}(k\phi_{1}, k\phi_{2}),1)\biggr)d\nu^{\alpha}(u)d\nu^{\beta}(v)\nonumber
\end{eqnarray}
where $f_{2k,\lambda}(b, \xi, t)$  is the infinite series \begin{eqnarray*} \label{f1} f_{2k,\lambda}(b, \xi, t)=\Gamma(k\lambda +1)\biggl(\frac{2}{b}\biggr)^{k\lambda} \sum_{j=0}^{\infty}\frac{(j+\lambda)}{\lambda} \mathcal{I}_{k(j+\lambda)}(bt)C_{j}^{(\lambda)}(\xi)
\end{eqnarray*}
with $C_{j}^{(\lambda)}(\xi)$  the Gegenbauer polynomial and $\mathcal{I}_{k}(b)$  the modified Bessel function of the first kind, i.e.
\[\mathcal{I}_{\nu}(b)=\sum_{n=0}^{\infty}\frac{(b/2)^{2n+\nu}}{n!\Gamma(n+\nu+1)}, \,b\in \mathbb{R} \]
as well as the symmetric beta measure
\begin{eqnarray*} d\nu^{\alpha}(u)&=&\frac{\Gamma(\alpha+1/2)}{\sqrt{\pi}\Gamma(\alpha)}(1-u^2)^{\alpha-1}du.\end{eqnarray*}
\end{theorem}

The series $f_{2k,\lambda}(b, \xi, t)$  admits  a closed form  in the Laplace domain as studied in our previous paper \cite{CDL}. By the inverse Laplace transform, it is realized that  $f_{2k, \lambda}(b, \xi, t)$ is in fact a Humbert $\Phi_{2}^{(m)}$ function, see \cite{CDL} formula (19). Later, this was proved more generally in \cite{DD1} using a different method. We summarize the results and present a more compact expression  using derivatives in the Laplace domain in the following lemma.

 \begin{lemma}\label{hl1}
For $k\ge 2$,
 the Laplace transform of $f_{2k,\lambda}(b, \xi, t)$ with respect to $t$ is given by
\begin{eqnarray}\label{bes}
\mathcal{L}[f_{2k,\lambda}(b,\xi, \cdot)](s)&=&\Gamma(k\lambda+1)\frac{2^{k\lambda}}{S}\frac{(S+s)^{k}-(s-S)^{k}}{((S+s)^{k}-2 b^{k}\xi+(s-S)^{k} )^{\lambda+1}}\nonumber\\
\nonumber\\&=&-\Gamma(k\lambda)2^{k\lambda}\frac{d}{ds}\left( \frac{ 1}{((S+s)^{k}-2 b^{k}\xi+(s-S)^{k} )^{\lambda}}\right),
\end{eqnarray}
where $S=\sqrt{s^2-b^2}$. Moreover, we  have \begin{eqnarray} \label{f2} f_{2k,\lambda}(b, \xi, 1)&=&\Phi_{2}^{(k)}\left(\lambda, \ldots, \lambda; k\lambda; b_{0},\ldots, b_{k-1}\right)\\
&=&e^{b_{0}}\Phi_{2}^{(k-1)}\left(\lambda, \ldots, \lambda; k\lambda; b_{1}-b_{0},\ldots, b_{k-1}-b_{0}\right)\nonumber
\end{eqnarray}
where $b_{j}=b\cos\left((q-2j\pi)/k\right)$,  $
 j=0, \ldots, k-1$  in which $q=\arccos (\xi).$
\end{lemma}

\begin{remark} The denominator $(S+s)^{k}-2b^{k}\xi+(s-S)^{k}$ is a polynomial in $s$ and satisfies the factorization
\begin{eqnarray}\label{fa1}
(S+s)^{k}- 2b^{k}\xi+(s-S)^{k}=2^{k}\prod_{l=0}^{k-1}\left(s-b \cos\left(\frac{q-2\pi l}{k}\right)\right).
\end{eqnarray}
 Formula (\ref{f2}) follows from the Laplace transform formula (\ref{l1}) and the above factorization (\ref{fa1}).
\end{remark}
\begin{remark} The first identity in (\ref{f2}) looks more symmetric than the second one. However, it is not possible to get an integral expression for $f_{2k, \lambda}(b, \xi, t)$ directly  from the integral representation (\ref{pi1}) of $\Phi_{2}^{(m)}$, due to the the validity  of the conditions there. This is not a problem for the second identity.
\end{remark}


Combining the Humbert function expression of $f_{2k, \lambda}(b, \xi, 1)$ in (\ref{f2})  and the integral expression (\ref{pi1}) for the function $\Phi_{2}^{(m)}$, we have the  following integral expression of the generalized Bessel function. From this integral expression, it is  directly seen that  the generalized Bessel function $\mathcal{J}_{\kappa}(x, y) $ is positive
 and that the complexified generalized Bessel function $\mathcal{J}_{\kappa}(ix, y)$ is bounded by $1$.
\begin{theorem}\label{fir} For $k\ge 2$,  the generalized Bessel function associated to the dihedral group $I_{2k}$ is given by
\begin{eqnarray}\label{mv1}\mathcal{J}_{\kappa}(z, w )&=&\frac{\Gamma{(k(\alpha+\beta))}}{2\Gamma(\alpha+\beta)^{k}}
\int_{-1}^{1}\int_{-1}^{1}\int_{T^{k-1}}\left(e^{\sum_{j=0}^{k-1}a_{j}^{+}t_{j}}+e^{\sum_{j=0}^{k-1}a_{j}^{-}t_{j}}\right) \nonumber\\&&\times\prod_{j=0}^{k-1}t_{j}^{\alpha+\beta-1}dt_{1}\ldots dt_{k-1}d\nu^{\alpha}(u)d\nu^{\beta}(v)
\end{eqnarray}
where  $a_{j}^{+}=|zw|\cos\left(\frac{q_{u, v}(k\phi_{1},  k\phi_{2})-2j\pi}{k}\right)$, $a_{j}^{-}=|zw|\cos\left(\frac{\pi-q_{u, v}(k\phi_{1},  k\phi_{2})-2j\pi}{k}\right),$   $
 j=0, \ldots, k-1$,
  $q_{u, v}(\phi_{1},  \phi_{2})=\arccos (\xi_{u,v}(\phi_{1}, \phi_{2}))$ and  $t_{0}=1-\sum_{j=1}^{k-1}t_{j}$.
\end{theorem}
\begin{remark} The integral over the simplex in (\ref{mv1}) is the integral expression of $f_{2k, \alpha+\beta}(|zw|, \xi_{u,v}(k\phi_{1}, k\phi_{2}), 1 )$. Although many variables $a_{j}^{\pm}$ appear, it is a function in the variables  $\xi_{u,v}(k\phi_{1}, k\phi_{2})$ and $|zw|$ by its series expansion, see also a direct verification  in \cite{DD1}. This further implies that $\mathcal{J}_{\kappa}(z, w)$ is a function in the variables $|zw|, \frac{{\rm Re}(z^{k}){\rm Re}({w}^{k})}{|zw|^{k}} $ and  $ \frac{{\rm Im} (z^{k}) {\rm Im} ({w}^{k})}{|zw|^{k}} $.
\end{remark}

It is known that the generalized Bessel function is the reproducing kernel of the invariant polynomials under $[\cdot, \cdot]_{\kappa}$.  By the series expansion of  $\mathcal{J}_{\kappa}(z, w)$, we obtain the following integral expressions for the reproducing kernels.
\begin{corollary} Let $p_{n}(z)$ be a  homogenous polynomial of degree $n$ invariant under the action of $I_{2k}, k\ge 2$, i.e. $p_{n}(z)=p_{n}(g\cdot z), $ for any $g\in I_{2k}$. Then the reproducing kernel for $p_{n}(z)$ is given by
\begin{eqnarray*}&&\mathcal{J}^{(n)}_{\kappa}(z, w )\\&=&\frac{\Gamma{(k(\alpha+\beta))}}{2\Gamma(n+1)\Gamma(\alpha+\beta)^{k}}
\int_{-1}^{1}\int_{-1}^{1}\int_{T^{k-1}}\left(\left(\sum_{j=0}^{k-1}a_{j}^{+}t_{j}\right)^{n}
+\left(\sum_{j=0}^{k-1}a_{j}^{-}t_{j}\right)^{n}\right) \nonumber\\&&\times\prod_{j=0}^{k-1}t_{j}^{\alpha+\beta-1}dt_{1}\ldots dt_{k-1}d\nu^{\alpha}(u)d\nu^{\beta}(v),
\end{eqnarray*}
and satisfies $p_{n}(z)=\left[p_{n}(w), \mathcal{J}^{(n)}_{\kappa}(z, w)\right]_{\kappa}$.
\end{corollary}
\begin{remark} The polynomials invariant under the group $I_{k}$ form an algebra.   This algebra is generated by $|z|^{2}$ and $\frac{z^{k}+\overline{z^{k}}}{2}$ (the Chevalley generators). As a linear space, the dimension of the invariant polynomials of degree $n$ can be determined by  Molien's generating function, see \cite{GB2}.
\end{remark}

The integral expression (\ref{mv1}) looks quite complicated, however, it reflects how the dihedral group $I_{2k}$ acts. Due to nature of the reduction method, it will be seen  in the following theorem that the expression is a composition of two operators, corresponding to the reflection and rotations in dihedral groups. Alternatively, by this expression, it is seen that $\mathcal{J}_{\kappa}(z, w)$ is $I_{2k}$ invariant.

\begin{theorem} \label{bess1} The generalized Bessel function $\mathcal{J}_{\kappa}(z, w )$ associated to $I_{2k}, k\ge 2$ is given by
\begin{eqnarray}\label{ba1}\mathcal{J}_{\kappa}(z, w)&=&\int_{-1}^{1}\int_{-1}^{1}J(w, |z|, s_{1}v, s_{2}u)d\nu^{\alpha}(u)d\nu^{\beta}(v)
\end{eqnarray}
with $s_{1}=\cos(k\phi_{1}), s_{2}=\sin(k\phi_{1})$ and \begin{eqnarray*}J(w, |z|, s_{1}, s_{2})&=&c_{\kappa, k}
\int_{T^{k-1}}\left(e^{\langle w, z\sum_{j=0}^{k-1}e^{-i2j\pi/k}t_{j}\rangle} +e^{\langle w, z\sum_{j=0}^{k-1}e^{i(2j-1)\pi/k}t_{j} \rangle}\right) \nonumber\\&&\times\prod_{j=0}^{k-1}t_{j}^{\alpha+\beta-1}dt_{1}\ldots dt_{k-1}\end{eqnarray*}
where  $t_{0}=1-\sum_{j=1}^{k-1}t_{j}$ and $c_{\kappa, k}=\frac{\Gamma{(k(\alpha+\beta))}}{2\Gamma(\alpha+\beta)^{k}}$.
\end{theorem}
\begin{proof} Note that the series $f_{2k, \lambda}(b, \xi, 1)$ is a function in variables $b$ and $\xi$. Hence, the function $J(w, |z|, s_{1}, s_{2})$ is a function in the variables $|zw|$ and $\xi_{1,1}(k\phi_{1}, k\phi_{2})$. The integrand  \begin{eqnarray*}J(w, |z|, s_{1}v, s_{2}u)&=&\frac{1}{2}\left[f_{2k, \lambda}(|zw|, \xi_{u,v}(k\phi_{1}, k\phi_{2}), 1)\right.\\&&+\left.f_{2k, \lambda}(|zw|, -\xi_{u,v}(k\phi_{1}, k\phi_{2}), 1)\right]\end{eqnarray*}
is a function in the variables $|zw|$ and $\xi_{u, v}(k\phi_{1}, k\phi_{2})$. The latter one can be obtained by  replacing $\xi_{1,1}(k\phi_{1}, k\phi_{2})$ in the former one by \[\xi_{u,v}(k\phi_{1}, k\phi_{2})=v\cos (k\phi_{1})\cos(k\phi_{2})+u\sin(k\phi_{1})\sin(k\phi_{2}).\] This fact and expressions (\ref{B1}) and  (\ref{mv1}) lead to the first identity.

 In the following, we simplify the function  \begin{eqnarray*}J(w, |z|, s_{1}, s_{2})&=&\frac{1}{2}(f_{2k, \lambda}(|z||w|, \xi_{1,1}(k\phi_{1}, k\phi_{2}), 1)\\&&+f_{2k, \lambda}(|z||w|, -\xi_{1,1}(k\phi_{1}, k\phi_{2}), 1)).\end{eqnarray*} It is seen that
\begin{eqnarray*}q_{1,1}(k\phi_{1}, k\phi_{2})&=&\arccos(\xi_{1,1}(k\phi_{1}, k\phi_{2}))\\&=&\arccos(\cos(k\phi_{1})\cos(k\phi_{2})+\sin(k\phi_{1})\sin(k\phi_{2}))\\
&=&k(\phi_{1}-\phi_{2}).
\end{eqnarray*}
This yields that
  \begin{eqnarray*}a_{j}^{+}&=&|zw|\cos\left(\frac{q_{1,1}(k\phi_{1}, k\phi_{2})-2j\pi}{k}\right)=|zw|\cos\left((\phi_{1}-\phi_{2})-\frac{2j\pi}{k}\right),\\
 a_{j}^{-}&=&|zw|\cos\left(\frac{\pi-2j\pi-q_{1,1}(k\phi_{1}, k\phi_{2})}{k}\right)=|zw|\cos\left(\frac{\pi-2j\pi}{k}-(\phi_{1}-\phi_{2})\right). \end{eqnarray*}
In other words, we have
 \begin{eqnarray*}a_{j}^{+}&=&{\rm Re}(w\overline{z}e^{j2i\pi/k})
=\langle w, ze^{-i2j\pi/k} \rangle,
 \\
 a_{j}^{-}&=& {\rm Re}(w\overline{ze^{i(2j-1)\pi/k}})=\left\langle w, ze^{i(2j-1)\pi/k} \right\rangle,\end{eqnarray*}
 where $\langle z, w\rangle={\rm Re} (z\overline{w})$ denote the usual Euclidean inner product for $z, w \in \mathbb{C}\cong \mathbb{R}^{2}$.  Hence for this special case $u=v=1$, the integrand of (\ref{mv1}) becomes
\begin{eqnarray}\label{fa2}&&e^{\sum_{j=0}^{k-1}a_{j}^{+}t_{j}}+e^{\sum_{j=0}^{k-1}a_{j}^{-}t_{j}}\\&=&e^{\langle w, z\sum_{j=0}^{k-1}e^{-i2j\pi/k}t_{j}\rangle} +e^{\langle w, z\sum_{j=0}^{k-1}e^{i(2j-1)\pi/k}t_{j} \rangle}.\nonumber\end{eqnarray}
The second formula is obtained.
\end{proof}

At the end of this subsection, we express $a_{j}^{\pm}$ by  Cartesian coordinates instead of the polar coordinates expression (\ref{fa2}).  We denote $(\sqrt[k]{z})_{j}, 0\le j\le k-1$ the $ k$ different $k$-th roots of $z$.

\begin{theorem} \label{bess} For $k\ge 2$,  the generalized Bessel function associated to the dihedral group $I_{2k}$ is given by
\begin{eqnarray*}\mathcal{J}_{\kappa}(z, w )&=&\frac{\Gamma{(k(\alpha+\beta))}}{2\Gamma(\alpha+\beta)^{k}}
\int_{-1}^{1}\int_{-1}^{1}\int_{T^{k-1}}h_{u,v, t_{0},\ldots, t_{k-1}}(z, w) \nonumber\\&&\times\prod_{j=0}^{k-1}t_{j}^{\alpha+\beta-1}dt_{1}\ldots dt_{k-1}d\nu^{\alpha}(u)d\nu^{\beta}(v)
\end{eqnarray*}
where  \begin{eqnarray*}&&h_{u,v, t_{0},\ldots, t_{k-1}}(z, w)\\&=&{\rm exp}\left(\sum_{j=0}^{k-1} t_{j}{\rm Re}\left(\sqrt[k]{{\rm Re}(\tilde{z}^{k}\overline{w^{k}})+i\sqrt{|zw|^{2k}-({\rm Re}(\tilde{z}^{k}\overline{w^{k}}) )^{2}}}\right)_{j} \right)\\
&&+{\rm exp}\left(\sum_{j=0}^{k-1} t_{j}{\rm Re}\left(\sqrt[k]{-{\rm Re}(\tilde{z}^{k}\overline{w^{k}})+i\sqrt{|zw|^{2k}-({\rm Re}(\tilde{z}^{k}\overline{w^{k}}) )^{2}}}\right)_{j} \right)
\end{eqnarray*}
here  $\tilde{z}^{k}=v{\rm Re} (z^{k})+ui {\rm Im} (z^{k}) $, $t_{0}=1-\sum_{j=1}^{k-1}t_{j}$ and
 $\sqrt{|zw|^{2k}-({\rm Re}(\tilde{z}^{k}\overline{w^{k}}) )^{2}}$ is the positive root.
\end{theorem}

\begin{proof} We only need to express $
e^{\sum_{j=0}^{k-1}a_{j}^{+}t_{j}}+e^{\sum_{j=0}^{k-1}a_{j}^{-}t_{j}}
$
in the integrand of (\ref{mv1}) by  Cartesian coordinates. It reduces to express $a_{j}^{\pm}$ by $z,w$ instead of the angles $q_{u,v}(k\phi_{1}, k\phi_{2})$.  In fact, it is seen that
\[a_{j}^{+}=|zw|\cos\left(\frac{q_{u, v}(k\phi_{1},  k\phi_{2})-2j\pi}{k}\right), 0\le j\le k-1\] are the real part of the $k$-roots of ${\rm Re}(\tilde{z}^{k}\overline{w^{k}})+i\sqrt{|zw|^{2k}-({\rm Re}(\tilde{z}^{k}\overline{w^{k}}) )^{2}}$ where $\tilde{z}^{k}=v{\rm Re} (z^{k})+i u{\rm Im} (z^{k}) $.
Hence,  we have
\begin{eqnarray*}e^{\sum_{j=0}^{k-1}a_{j}^{+}t_{j}}={\rm exp}\left(\sum_{j=0}^{k-1} t_{j}{\rm Re}\left(\sqrt[k]{{\rm Re}(\tilde{z}^{k}\overline{w^{k}})+i\sqrt{|zw|^{2k}-({\rm Re}(\tilde{z}^{k}\overline{w^{k}}) )^{2}}}\right)_{j} \right).
\end{eqnarray*}
The expression for $a_{j}^{-}$ is obtained similarly.
\end{proof}
\begin{remark} The result will not be changed if we choose $\sqrt{|zw|^{2k}-({\rm Re}(\tilde{z}^{k}\overline{w^{k}}) )^{2}}$ as the negative square root. When $u=v=1$, the $k$-th roots of ${\rm Re}(\tilde{z}^{k}\overline{w^{k}})+i\sqrt{|zw|^{2k}-({\rm Re}(\tilde{z}^{k}\overline{w^{k}}) )^{2}}$ becomes $\sqrt[k]{z^{k}\overline{w^{k}}}$. Hence, it reduces to  the formula (\ref{fa2}) \begin{eqnarray*}a_{j}^{+}&=&{\rm Re}\left(\sqrt[k]{z^{k}\overline{w^{k}}}\right)_{j}
=\langle w, ze^{-i2j\pi/k} \rangle.\end{eqnarray*} However, it is not possible to express $a_{j} $ as an inner product of the form $\langle w, g(z)\rangle$ for general $u$ and $v$.
\end{remark}
\begin{remark} For the odd order dihedral group $I_{k}$, the generalized Bessel function can be obtained from $I_{2k}$, see \cite{Dn}.
\end{remark}

\subsection{The Dunkl kernel}
In this section, we give two kinds of expressions of the Dunkl kernel  based on the Laplace domain result obtained in \cite{CDL}.

 We still identify $\mathbb{R}^{2}$ with the complex plane $\mathbb{C}$, and denote $z=|z|e^{i\phi_{1}}$, $w=|w|e^{i\phi_{2}}$.
The Laplace domain result is obtained by introducing an auxiliary variable in the series expansion of the Dunkl kernel and then taking the Laplace transform. More precisely, the series is
  \begin{eqnarray*}
E_{\kappa}(z,w, t)=\frac{2^{\alpha+\beta}\Gamma(k(\alpha+\beta)+1)}{|zw|^{k(\alpha+\beta)}}
\sum_{j=0}^{\infty}\mathcal{I}_{j+k(\alpha+\beta)}(|zw|t) P_{j}\left(I_{k}; e^{i\phi_{1}}, e^{i\phi_{2}}\right),
\end{eqnarray*}
where  $P_{j}(I_{k}; e^{i\phi_{1}}, e^{i\phi_{2}})$ is the reproducing kernel of the Dunkl harmonics of degree $j$ and $\mathcal{I}_{j}$ is the modified Bessel function of first kind, see \cite{SKO} for the series expansion of the general Dunkl kernel.
It is proved that the above infinite series admits a closed expression in the Laplace domain for the dihedral groups.
\begin{theorem}\cite{CDL}\label{el1} For the even dihedral group $I_{2k}$, the Laplace transform of the Dunkl kernel $E_{\kappa}(z, w, t)$  with respect to $t$ is given by
\begin{eqnarray*}
\mathcal{L}(E_{\kappa}(z, w,t))&=&\Gamma(k(\alpha+\beta) +1)
\int_{-1}^{1}\int_{-1}^{1}f_{I_{2k}}(s, z, w)d\mu^{\alpha}(u)d\mu^{\beta}(v)
\end{eqnarray*}
where
\begin{eqnarray*}
f_{I_{2k}}(s, z, w)=\frac{2^{k(\alpha+\beta)}}{(s-{\rm Re}(z\overline{w}))}\frac{(s+S)^{k}-2{\rm Re}(z^{k}\overline{w^{k}})+(s-S)^{k}}{((s+S)^{k}-2|zw|^{k}\xi_{u,v}(k\phi_{1}, k\phi_{2})+(s-S)^{k})^{\alpha+\beta+1}}
\end{eqnarray*}
with $S=\sqrt{s^{2}-|zw|^{2}}$ and  \[d\mu^{\gamma}(\omega)=\frac{\Gamma(\gamma+1/2)}{\Gamma(1/2)\Gamma( \gamma)}(1+\omega)(1-\omega^{2})^{\gamma-1}d\omega.\]
\end{theorem}

\begin{remark} The substitution of $x$ by $ix$ in the original formula of Theorem 12 in \cite{CDL} has been made here.  For $x, y\in \mathbb{R}^{2},$  the Dunkl kernel studied in \cite{CDL}
is in fact
\[E(-ix, y)=V_{\kappa}[e^{-i\langle \cdot, y\rangle}](x).\]
However, the Dunkl kernel studied here is
\[E(x, y)=V_{\kappa}[e^{\langle \cdot, y\rangle}](x).\]
\end{remark}
\begin{remark} The integrand $f_{I_{2k}}(s, z, w)$ can be  factored as follows, see  Lemma 3 in \cite{CDL},
 \begin{eqnarray*}
 f_{I_{2k}}(s, z, w)=
\frac{A(s,|zw|, q_{1,1}(k\phi_{1}, k\phi_{2}))}{B(s, z, w)[A(s,|zw|, q_{u,v}(k\phi_{1}, k\phi_{2}))]^{\alpha+\beta+1}},
 \end{eqnarray*}
 where $B(s, z, w)=s-{\rm Re}(z\overline{w})$ and
 \begin{eqnarray*}
A(s,|zw|, q_{u,v}(k\phi_{1}, k\phi_{2}))&=&\prod_{\ell=0}^{k-1}\biggl(s-|zw|\cos\biggl(\frac{q_{u,v}(k\phi_{1}, k\phi_{2}+2\pi \ell}{k}\biggr)\biggr).
\end{eqnarray*}

\end{remark}

Similar to the generalized Bessel function, by the Laplace transform formula (\ref{l1}), the Dunkl kernel can be expressed using the Humbert function and  integrals over the simplex.

\begin{theorem}\label{m2}
 For each dihedral group $I_{2k}$ and non-negative multiplicity function $\kappa=(\alpha, \beta)$, the Dunkl kernel is given by
\begin{eqnarray*}
E_{\kappa}(z, w)&=&
\int_{-1}^{1}\int_{-1}^{1}\left[h_{\alpha+\beta}( z, w, u, v)+\frac{2^{1-k}\Gamma(k(\alpha+\beta)+1)}{\Gamma(k(\alpha+\beta+1)+1)}|zw|^{k}\right.\\ &&\times\xi_{u-1,v-1}(k\phi_{1}, k\phi_{2}) h_{\alpha+\beta+1}( z, w, u, v)\biggr]d\mu^{\alpha}(u)d\mu^{\beta}(v),\end{eqnarray*}
where
\begin{eqnarray*}
h_{\gamma}(z, w, u, v)&=&\Phi_{2}^{(k+1)}(\gamma,\ldots, \gamma, 1; k\gamma+1; a_{0},\ldots, a_{k-1}, a_{k})\\
&=& e^{a_{k}}\Phi_{2}^{(k)}(\gamma,\ldots, \gamma; k\gamma+1; a_{0}-a_{k},\ldots, a_{k-1}-a_{k}) \end{eqnarray*}
with $a_{j}=|zw|\cos\left(\frac{q_{u,v}(k\phi_{1}, k\phi_{2})+2\pi j}{k}\right)$, $j=0,\ldots k-1$ and $a_{k}={\rm Re} \,(z\overline{w}) $.
This can be equivalently expressed as
\begin{eqnarray*}
E_{\kappa}(z, w)&=&\frac{\Gamma(k(\alpha+\beta)+1)}{\Gamma(\alpha+\beta)^{k}}
\int_{-1}^{1}\int_{-1}^{1}\int_{T^{k}}e^{\sum_{j=0}^{k}a_{j}t_{j}}
d\omega_{\alpha+\beta}d\mu^{\alpha}(u)d\mu^{\beta}(v)
\\&&+\frac{2^{1-k}\Gamma(k(\alpha+\beta)+1)}{\Gamma(\alpha+\beta+1)^{k}}
\int_{-1}^{1}\int_{-1}^{1}\int_{T^{k}}|zw|^{k} \xi_{u-1, v-1}(k\phi_{1}, k\phi_{2})\\ &&\times e^{\sum_{j=0}^{k}a_{j}t_{j}}
d\omega_{\alpha+\beta+1} d\mu^{\alpha}(u)d\mu^{\beta}(v)
\end{eqnarray*}
where $t_{k}=1-\sum_{j=0}^{k-1}t_{j}$ and $d\omega_{\nu}=\prod_{j=0}^{k-1}t_{j}^{\nu-1}dt_{0}\ldots dt_{k-1}$.
\end{theorem}
\begin{proof}  We split $f_{I_{2k}}(s, z, w)$  into two parts
\begin{eqnarray}\label{fs}f_{I_{2k}}(s, z, w)&=&\frac{1}{B(s, z, w)[A(s,|zw|, q_{u,v}(k\phi_{1}, k\phi_{2}))]^{\alpha+\beta}}\\&&+ \frac{|zw|^{k}\xi_{u-1,v-1}(k\phi_{1}, k\phi_{2})}{2^{k-1}B(s, z, w)[A(s,|zw|, q_{u,v}(k\phi_{1}, k\phi_{2}))]^{\alpha+\beta+1}}.\nonumber \end{eqnarray}
Then the first expression follows from taking the inverse Laplace transform for each term using the $\Phi_{2}^{(m)}$ functions and then putting $t=1$.

The second formula is obtained by replacing the Humbert function by its integral expression (\ref{pi1}), which is similar to the integral expression (\ref{mv1}) for the generalized Bessel function.
\end{proof}

\begin{ex} For the dihedral group $I_{2}$, we have
\begin{eqnarray*}
E_{\kappa}(z, w)&=&
\int_{-1}^{1}\int_{-1}^{1}\left[h_{\alpha+\beta}( z, w, u, v)+\frac{|zw|\xi_{u-1,v-1}(\phi_{1}, \phi_{2})}{\alpha+\beta+1}\right.\\&&\times h_{\alpha+\beta+1}(z, w, u, v))\biggr]d\mu^{\alpha}(u)d\mu^{\beta}(v),\end{eqnarray*} where
\[h_{\alpha}(z, w, u, v))=e^{{\rm Re}\,(z\overline{w}) }
\sum_{j=0}^{\infty}\frac{(\alpha)_{j}}{(\alpha+1)_{j}}\frac{A^{j}}{j!}   \]
in which \[A=|zw|((v-1)\cos\phi_{1}\cos\phi_{2}+(u-1)\sin \phi_{1}\sin\phi_{2}  ).\] Direct computation shows that the integrand reduces to $e^{|zw|\cos(q_{u,v}(\phi_{1}, \phi_{2}))}$. Hence, the Dunkl kernel for the root system $I_{2}$ is
\begin{eqnarray*}
E_{\kappa}(z, w)&=&
\int_{-1}^{1}\int_{-1}^{1}e^{|zw| (v\cos\phi_{1}\cos\phi_{2}+u\sin\phi_{1}\sin\phi_{2})}d\mu^{\alpha}(u)d\mu^{\beta}(v)\end{eqnarray*}
which coincides with the  known results.
\end{ex}

In the following, we derive the second expression for the Dunkl kernel, which is more compact. We start by rewriting the Laplace domain expression.

\begin{theorem} \label{b12}
 For the even dihedral group $I_{2k}$, the Laplace transform of the Dunkl kernel $E_{\kappa}(z, w, t)$  with respect to $t$ is given by
\begin{eqnarray*}
&&\mathcal{L}(E_{\kappa}(z, w,t))\\&=&\Gamma(k(\alpha+\beta) +1)
\int_{-1}^{1}\int_{-1}^{1}\frac{1}{B(s, z, w)[A(s, |zw|, q_{u,v}(k\phi_{1}, k\phi_{2}))]^{\alpha+\beta}}\\
&&\times\left[(1+u)(1+v)-\frac{2}{\alpha+\beta}
(\alpha u(1+v)+\beta v(1+u))\right]d\nu^{\alpha}(u)d\nu^{\beta}(v),
\end{eqnarray*}
where $B(s, z, w)=s-{\rm Re}(z\overline{w})$ and
 \begin{eqnarray*}
A(s,|zw|, q_{u,v}(k\phi_{1}, k\phi_{2}))&=&\prod_{\ell=0}^{k-1}\biggl(s-|zw|\cos\biggl(\frac{q_{u,v}(k\phi_{1}, k\phi_{2}+2\pi \ell}{k}\biggr)\biggr).
\end{eqnarray*}
\end{theorem}
\begin{proof} In (\ref{fs}),  $f_{I_{2k}}(s, z, w)$  is split into two parts. The second part of $(\ref{fs})$ satisfies
\begin{eqnarray*}
&&\frac{|zw|^{k}\xi_{u-1,v-1}(k\phi_{1}, k\phi_{2})}{B(s, z, w)[A(s,|zw|, q_{u,v}(k\phi_{1}, k\phi_{2}))]^{\alpha+\beta+1}}\\
&=&\frac{|zw|^{k}\xi_{u-1,v-1}(k\phi_{1}, k\phi_{2})}{B(s, z, w)\left[\frac{1}{2^{k}}((S+s)^{k}- 2|zw|^{k}\xi_{u,v}(k\phi_{1}, k\phi_{2}) +(s-S)^{k})\right]^{\alpha+\beta+1}}\nonumber\\
&=&\frac{2^{k-1}}{\alpha+\beta}\frac{1}{B(s, z, w)}\left((u-1)\frac{d}{du} \frac{1 }{[A(s, |zw|, q_{u,v}(k\phi_{1}, k\phi_{2})]^{\alpha+\beta}}\right.\\&&+\left.
(v-1)\frac{d}{dv}\frac{1 }{[A(s,|zw|, q_{u,v}(k\phi_{1}, k\phi_{2}))]^{\alpha+\beta}}\right)
\end{eqnarray*}
where  as before \[\xi_{u,v}(k\phi_{1}, k\phi_{2}) =v\cos(k\phi_{1})\cos(k\phi_{2})+u\sin(k\phi_{1})\sin(k\phi_{2}). \]
This leads to the following expression for the Dunkl kernel in the Laplace domain
\begin{eqnarray}\label{ld1}
&&\frac{1}{\Gamma(k(\alpha+\beta) +1)}\mathcal{L}(E_{\kappa}(z, w,t)) \\
&=&\int_{-1}^{1}\int_{-1}^{1}\frac{1}{B(s, z, w)[A(s, |zw|, q_{u,v}(k\phi_{1}, k\phi_{2}))]^{\alpha+\beta}}d\mu^{\alpha}(u)d\mu^{\beta}(v)\nonumber\\
&&+\frac{1}{\alpha+\beta}\frac{1}{B(s, z,w)}\biggl[\int_{-1}^{1}\int_{-1}^{1}\biggl[\frac{d}{du} \left(\frac{1 }{[A(s,|zw|, q_{u,v}(k\phi_{1}, k\phi_{2}))]^{\alpha+\beta}}\right)\nonumber\\ &&\times (u-1)
+\frac{d}{dv} \left(\frac{1 }{[A(s,|zw|, q_{u,v}(k\phi_{1}, k\phi_{2}))]^{\alpha+\beta}}\right)(v-1)\biggr]d\mu^{\alpha}(u)d\mu^{\beta}(v)\biggr].\nonumber
\end{eqnarray}
The second integral in (\ref{ld1}) is further simplified using integration by parts
\begin{eqnarray*}&&\int_{-1}^{1}\int_{-1}^{1}\frac{d}{du} \left(\frac{1 }{[A(s, |zw|, q_{u,v}(k\phi_{1}, k\phi_{2}))]^{\alpha+\beta}}\right)(u-1)d\mu^{\alpha}(u)d\mu^{\beta}(v)\\
&=&\int_{-1}^{1}\int_{-1}^{1}\frac{1 }{[A(s,|zw|, q_{u,v}(k\phi_{1}, k\phi_{2}))]^{\alpha+\beta}}\left(-\frac{d}{du}(u-1)d\mu^{\alpha}(u)d\mu^{\beta}(v)\right)\\
&=&\frac{-2\alpha\Gamma(\alpha+1/2)}{\Gamma(1/2)\Gamma(\alpha)}\int_{-1}^{1}\int_{-1}^{1}\frac{1 }{[A(s,|zw|, q_{u,v}(k\phi_{1}, k\phi_{2}))]^{\alpha+\beta}}u(1-u^{2})^{\alpha-1}dud\mu^{\beta}(v)
.\end{eqnarray*}
The third integral is simplified similarly.
Collecting all, we obtain the desired result.
\end{proof}

 By the inverse Laplace transform and then setting $t=1$, Lemma \ref{b12}  leads to the following new expression of the Dunkl kernel. The proof is omitted here.
\begin{theorem} \label{ib111}
For each dihedral group $I_{2k}$ and positive multiplicity function $\kappa$, the Dunkl kernel is given by
\begin{eqnarray*}
E_{\kappa}(z, w)&=&
\int_{-1}^{1}\int_{-1}^{1}\left[(1+u)(1+v)-\frac{2}{\alpha+\beta}
(\alpha u(1+v)+\beta v(1+u))\right]\\&& \times h_{\alpha+\beta}( z, w, u, v)d\nu^{\alpha}(u)d\nu^{\beta}(v)
.\end{eqnarray*}
For each odd dihedral group $I_{k}$ and positive multiplicity function $\kappa$, the Dunkl kernel is
\begin{eqnarray*}E_{\kappa}(z, w)&=&\int_{-1}^{1}h_{\alpha}( z, w, u, 1)(1- u)d\nu^{\alpha}(u).\end{eqnarray*}
In these formulas the expression
\begin{eqnarray*}
h_{\gamma}(z, w, u, v)&=&\Phi_{2}^{(k+1)}(\gamma,\ldots, \gamma, 1; k\gamma+1; a_{0},\ldots, a_{k-1}, a_{k})\\
&=& e^{a_{k}}\Phi_{2}^{(k)}(\gamma,\ldots, \gamma; k\gamma+1; a_{0}-a_{k},\ldots, a_{k-1}-a_{k}) \end{eqnarray*} is defined in Theorem \ref{m2}.
\end{theorem}

The integrand $h_{\gamma}(z, w, u, v)$ in Theorem \ref{ib111} is positive, by its integral expression over the simplex.
In the following, we will show that the measure in the integral is positive as well. This further implies that the Dunkl kernel satisfies $E_{\kappa}(z, w)> 0$.
\begin{lemma}
For $u, v\in [-1,1]$ and $\alpha, \beta\ge 0$, we have
\begin{eqnarray}\label{in1}(1+u)(1+v)-\frac{2}{\alpha+\beta}
(\alpha u(1+v)+\beta v(1+u))\ge 0.\end{eqnarray}
\end{lemma}
\begin{proof}
When $u\le 0$ and $v\le 0$, the inequality (\ref{in1}) holds obviously. When $u\le 0$ and $v\ge 0$,
we have $-\alpha u(1+v)\ge 0 $ and
\begin{eqnarray*}(1+u)(1+v)-\frac{2\beta}{\alpha+\beta}
 v(1+u)\ge  (1+u)\left(2v-\frac{2\beta}{\alpha+\beta}
 v\right) \ge 0
 .\end{eqnarray*} Therefore, the inequality (\ref{in1}) holds as well. The case when $u\ge 0$ and $v\le 0$ is similar, we will not repeat the proof.

When $u\ge 0$ and $v\ge0$, we have $1+u\ge 2u$ and $1+v\ge 2v$. In this case, we consider the quotient
\begin{eqnarray*}
\frac{2}{\alpha+\beta}\frac{
(\alpha u(1+v)+\beta v(1+u))}{(1+u)(1+v)}=
\frac{2\alpha}{\alpha+\beta}\frac{u}{1+u}+ \frac{2\beta}{\alpha+\beta}\frac{v}{1+v}
\le1
\end{eqnarray*}
which completes the proof.
\end{proof}

Since the intertwining operator preserves  homogeneous polynomials,  we obtain an integral for the reproducing kernel of homogenous polynomials by a series expansion.
\begin{corollary} Denote by $\mathcal{P}_{n}$ the space of homogenous polynomial of degree $n$. For dihedral group $ I_{2k}$, the reproducing kernel of $\mathcal{P}_{n}$ is given by
\begin{eqnarray*} &&V_{\kappa}\left(\frac{\langle \cdot, w\rangle^{n}}{n!}\right)(z)
=
\frac{\Gamma(k(\alpha+\beta)+1)}{\Gamma(\alpha+\beta)^{k}}
\int_{-1}^{1}\int_{-1}^{1}\int_{T^{k}}\left(\sum_{j=0}^{k}a_{j}t_{j}\right)^{n}
\\&&\times\left[(1+u)(1+v)-\frac{2}{\alpha+\beta}
(\alpha u(1+v)+\beta v(1+u))\right]d\omega_{\alpha+\beta}d\nu^{\alpha}(u)d\nu^{\beta}(v).
\end{eqnarray*}
Similarly, for the odd dihedral group $I_{k}$, we have
\begin{eqnarray*} &&V_{\kappa}\left(\frac{\langle \cdot, w\rangle^{n}}{n!}\right)(z)
=
\frac{\Gamma(k\alpha+1)}{\Gamma(\alpha)^{k}}
\int_{-1}^{1}\int_{T^{k}}\left(\sum_{j=0}^{k}a_{j}t_{j}\right)^{n}
(1-u)d\omega_{\alpha+\beta}d\nu^{\alpha}(u)
\end{eqnarray*}
where
$a_{j}=|zw|\cos\left(\frac{q_{u,v}(k\phi_{1}, k\phi_{2})+2\pi j}{k}\right)$, $j=0,\ldots k-1$, $a_{k}={\rm Re} \,(z\overline{w}) $, $t_{k}=1-\sum_{j=0}^{k-1}t_{j}$ and $d\omega_{\nu}=\prod_{j=0}^{k-1}t_{j}^{\nu-1}dt_{0}\ldots dt_{k-1}$.
\end{corollary}
\subsection{The intertwining operator }
 The expressions of the generalized Bessel function and the Dunkl kernel together with the result for general root systems, i.e. Theorem \ref{gin}, lead to integral expressions for the intertwining operator.

  We consider  the intertwining operator associated to invariant polynomials first.
Recall that the generalized Bessel function is an integral of the following Humbert function, see Theorem \ref{bess1},  \begin{eqnarray*}J(w, |z|, s_{1}v, s_{2}u)&=&\frac{1}{2} \left(\Phi_{2}^{(k)}\left(\lambda, \ldots, \lambda; k\lambda; a^{+}_{0},\ldots, a^{+}_{k-1}\right)\right.\\&&+\left.\Phi_{2}^{(k)}\left(\lambda, \ldots, \lambda; k\lambda; a^{-}_{0},\ldots, a^{-}_{k-1}\right)\right)\end{eqnarray*}
where $s_{1}=\cos(k\phi_{1})$, $s_{2}=\sin(k\phi_{1})$, $a_{j}^{+}=|zw|\cos\left(\frac{q_{u, v}(k\phi_{1},  k\phi_{2})-2j\pi}{k}\right)$, $a_{j}^{-}=|zw|\cos\left(\frac{\pi-q_{u, v}(k\phi_{1},  k\phi_{2})-2j\pi}{k}\right),$   $
 j=0, \ldots, k-1$.

 Combining Corollary \ref{cor} with Theorem \ref{bess1}, then yields,
\begin{theorem} \label{intn1}Let $p(x)$ be a polynomial invariant under the action of $I_{2k}, k\ge 2$, i.e. $p(x)=p(g\cdot x), $ for any $g\in I_{2k}$. Then, the intertwining operator $V_{\kappa}$ associated to $I_{2k}$ and $\kappa=(\alpha, \beta)$ is given by
\begin{eqnarray*}
V_{\kappa}(p)(z)=\int_{-1}^{1}\int_{-1}^{1}p_{u,v}(z)
d\nu^{\alpha}(u)d\nu^{\beta}(v)
\end{eqnarray*}
where \begin{eqnarray*}&&p_{u,v}(z)=\left[p(w), J(w, |z|, s_{1}v, s_{2}u)\right]_{0}\\&=&
\frac{1}{2\pi}\int_{\mathbb{R}^{2} } \mathcal{F}\left(e^{-|\cdot|^{2}/2}J(z, \cdot, u, v)\right)(y)\mathcal{F}\left(e^{-|\cdot|^{2}/2} p(\cdot)\right)(y)e^{|y|^{2}/2}dy.\end{eqnarray*}
This is equivalently expressed as
\begin{eqnarray*} &&V_{\kappa}(p)(z)\\&=&\frac{\Gamma{(k(\alpha+\beta))}}{2\Gamma(\alpha+\beta)^{k}}
\int_{-1}^{1}\int_{-1}^{1}\int_{T^{k-1}} \left[p(w), \left(e^{\sum_{j=0}^{k-1}a_{j}^{+}t_{j}}+e^{\sum_{j=0}^{k-1}a_{j}^{-}t_{j}}\right)\right]_{0} \nonumber\\&&\times\prod_{j=0}^{k-1}t_{j}^{\alpha+\beta-1}dt_{1}\ldots dt_{k-1}d\nu^{\alpha}(u)d\nu^{\beta}(v)
\end{eqnarray*}
where $a_{j}^{\pm}$ is defined in Theorem \ref{fir}.
 \end{theorem}

\begin{ex} For the group $I_{2}$, denote $x=(x_{1}, x_{2})$ and $y=(y_{1}, y_{2})$.  By setting $k=1$ in (\ref{f2}),  the generalized Bessel function associated to the group $I_{2}$ is given by \begin{eqnarray*}\mathcal{J}_{\kappa}(x, y)=\frac{1}{2}\int_{-1}^{1}\int_{-1}^{1}\left(e^{vx_{1}y_{1}+ux_{2}y_{2}}+e^{-(vx_{1}y_{1}+ux_{2}y_{2})}\right)d\nu^{\alpha}(u)d\nu^{\beta}(v). \end{eqnarray*}
By Theorem \ref{intn1},  the intertwining operator for the $I_{2}$-invariant polynomials is given by
\begin{eqnarray*}&&V_{\kappa}(p)(x)\\
&=&\frac{1}{2}\int_{-1}^{1}\int_{-1}^{1}\left[p(x),  e^{vx_{1}y_{1}+ux_{2}y_{2}}+e^{-(vx_{1}y_{1}+ux_{2}y_{2})}\right]_{0}d\nu^{\alpha}(u)d\nu^{\beta}(v)\\
&=&\frac{1}{2}\int_{-1}^{1}\int_{-1}^{1}\left(p(vx_{1}, ux_{2})+p(-vx_{1}, -ux_{2})\right)d\nu^{\alpha}(u)d\nu^{\beta}(v)\\&=&\int_{-1}^{1}\int_{-1}^{1}p(vx_{1}, ux_{2})d\nu^{\alpha}(u)d\nu^{\beta}(v) \end{eqnarray*}
which coincides with the known results.
\end{ex}
We study  the intertwining operator for the root system $B_{2}$ again.
\begin{ex} (Root system $B_{2}$)  For $x=(x_{1}, x_{2}), y=(y_{1}, y_{2})$ and $\kappa=(\alpha, \beta)$, by setting $k=2$ in Theorem \ref{bess1}, the generalized Bessel function takes the following form in
Cartesian coordinates, (see also \cite{AD}, \cite{CDL}, \cite{Dn})
\begin{eqnarray*}
\mathcal{J}_{\kappa}(x,y)=\int_{-1}^{1}\int_{-1}^{1}
\tilde{\mathcal{I}}_{\alpha+\beta-1/2}\left(\sqrt{\frac{Z_{x,y}(u,v)}{2}}\right)d\nu_{\alpha}(u)d\nu_{\beta}(v)
\end{eqnarray*}
where
\[\tilde{\mathcal{I}}_{\nu}(t)=\Gamma(\nu+1)\sum_{n=0}^{\infty}\frac{(t/2)^{2n}}{n!\Gamma(n+\nu+1)}\]
and
 \[Z_{x,y}(u,v)=(x_{1}^{2}+x_{2}^{2})(y_{1}^{2}+y_{2}^{2})+u(x_{1}^2-x_{2}^2)
(y_{1}^2-y_{2}^2)+4vx_{1}x_{2}y_{1}y_{2}. \]
In order to obtain the intertwining operator for the invariant polynomials, we only need to compute  $p_{u,v}(y)$  defined in Theorem \ref{intn1}, i.e.
\begin{eqnarray*}
p_{u,v}(y)=\left[p(x), \tilde{\mathcal{I}}_{\alpha+\beta-1/2}\left(\sqrt{\frac{Z_{x,y}(u,v)}{2}}\right) \right]_{0}.\end{eqnarray*}
By the  Mehler-Sonine type integral expression of the Bessel function  and the reproducing property of the exponential, we have
\begin{eqnarray*}
&&p_{u,v}(y)\\&=&c_{B_{2}}\left[p(x),
\int_{\{t_{1}^{2}+t_{2}^{2}\le 1\}}e^{x_{1}at_{1}+ax_{2}(ct_{1}+bt_{2})}(1-t_{1}^{2}-t_{2}^{2})^{\alpha+\beta-3/2}dt_{1}dt_{2}\right]_{0}
\\&=&c_{B_{2}}\int_{\{t_{1}^{2}+t_{2}^{2}\le 1\}}p(at_{1}, a(ct_{1}+bt_{2}))(1-t_{1}^{2}-t_{2}^{2})^{\alpha+\beta-3/2}dt_{1}dt_{2}
\end{eqnarray*}
where $c_{B_{2}}=(\alpha+\beta-1/2)/\pi$,  $a, b$ and $c$ have been determined explicitly in \cite{AD} as
\begin{eqnarray*}
a&=&\left(\frac{y_{1}^{2}+y_{2}^{2}+u(y_{1}^{2}-y_{2}^{2})}{2}\right)^{1/2},
\\b&=&\frac{\left((y_{1}^{2}-y_{2}^{2})^{2}(1-u^{2})+4y_{1}^{2}y_{2}^{2}(1-v^{2})\right)^{1/2}}{y_{1}^{2}+y_{2}^{2}+u(y_{1}^{2}-y_{2}^{2})},
\\c&=&\frac{2vy_{1}y_{2}}{y_{1}^{2}+y_{2}^{2}+u(y_{1}^{2}-y_{2}^{2})}.
\end{eqnarray*}
Hence, for $\alpha+\beta>1/2$ and  $I_{4}$-invariant polynomial $p(y)$, the intertwining operator associated to $B_{2}$ is given by
\begin{eqnarray*}
V_{\kappa}(p)(y)&=&c_{B_{2}}\int_{-1}^{1}\int_{-1}^{1}\int_{\{t_{1}^{2}+t_{2}^{2}\le 1\}}p(at_{1}, a(ct_{1}+bt_{2}))\\&& \times(1-t_{1}^{2}-t_{2}^{2})^{\alpha+\beta-3/2}dt_{1}dt_{2}
d\nu^{\alpha}(u)d\nu^{\beta}(v).
\end{eqnarray*}
It is seen that the measure given above is positive, therefore the integral transform is a positive operator as expected.
\end{ex}


Let us now turn to the general case, which follows by combining Theorem \ref{gin} with Theorem \ref{ib111}.
\begin{theorem} \label{mt} For  polynomials $p(z)$, the intertwining operator $V_{\kappa}$ for the dihedral group $I_{2k}$ is given by
\begin{eqnarray*}
V_{\kappa}(p)(z)&=&\int_{-1}^{1}\int_{-1}^{1}\left((1+u)(1+v)-\frac{2}{\alpha+\beta}
(\alpha u(1+v)+\beta v(1+u))\right)\\ &&\times P_{\alpha+\beta}(z, u, v)d\nu^{\alpha}(u)d\nu^{\beta}(v).
\end{eqnarray*}
The intertwining operator $V_{\kappa}$ for the odd dihedral group $I_{k}$  is given by
\begin{eqnarray*}
V_{\kappa}(p)(z)&=&\int_{-1}^{1}P_{\alpha}(z, u, v) (1-u)d\nu^{\alpha}(u).
\end{eqnarray*}
In these formulas, we put
\begin{eqnarray*}
&&P_{\gamma}(z, u, v) )=\left[p(w), h_{\gamma}(z, w, u, v)\right]_{0}\\
&=&\frac{1}{2\pi}\int_{\mathbb{R}^{2} } \mathcal{F}\left(e^{-|\cdot|^{2}/2}h_{\gamma}(z, \cdot, u, v)\right)(y)\mathcal{F}\left(e^{-|\cdot|^{2}/2} p(\cdot)\right)(y)e^{|y|^{2}/2}dy
\end{eqnarray*}
with
\begin{eqnarray*}
h_{\gamma}(z, w, u, v)&=&\Phi_{2}^{(k+1)}(\gamma,\ldots, \gamma, 1; k\gamma+1; a_{0},\ldots, a_{k-1}, a_{k}) \end{eqnarray*} defined in Theorem \ref{m2}.
\end{theorem}

We verify this directly for the group $I_{1}$.
\begin{ex} Let $x=(x_{1}, x_{2})$ and $y=(y_{1}, y_{2})$. For the rank one case, the intertwining operator is given by \begin{eqnarray*}
V_{\kappa}(p)(x)&=&\int_{-1}^{1}P_{\alpha}(x, u, v) (1-u)d\nu^{\alpha}(u)
\end{eqnarray*}
where
\begin{eqnarray*}
 P_{\gamma}(x, u, v) )&=&\left[p(y), h_{\gamma}(x, y, u, v)\right]_{0}\\
&=&\alpha\int_{0}^{1} [p(y), e^{x_{1}y_{1}+(1+ut-t)x_{2}y_{2}} ]_{0} t^{\alpha-1}dt.
\end{eqnarray*}
Computing the Fischer inner product, the intertwining operator is expressed as
\begin{eqnarray*}
V_{\kappa}(p)(x_{1}, x_{2})&=&\alpha\int_{-1}^{1}\int_{0}^{1} p(x_{1}, (ut+1-t)x_{2})t^{\alpha-1}dt (1-u)d\nu^{\alpha}(u)\\
&=&\alpha\int_{-1}^{1}\int_{0}^{1} p(x_{1}, (ut+1-t)x_{2})t^{\alpha-1}dtd\mu^{\alpha}(u)
\\&&+\alpha
\int_{-1}^{1}\int_{0}^{1} p(x_{1}, (ut+1-t)x_{2})t^{\alpha-1}
(-2u)dtd\nu^{\alpha}(u).
\end{eqnarray*}
In the following, we only consider the polynomial $x_{2}^{n}$, because in the present case the intertwining operator has no influence on the  variable $x_{1}$. We claim that the following equality holds
\begin{eqnarray} \label{q1}
&&\alpha
\int_{-1}^{1}\int_{0}^{1} (ut+1-t)^{n}t^{\alpha-1}dt
(-2u)d\nu^{\alpha}(u)\\
&=&\int_{-1}^{1}\int_{0}^{1}n(u-1)(ut+1-t)^{n-1}t^{\alpha}dtd\mu^{\alpha}(u).\nonumber
\end{eqnarray}
The identity (\ref{q1}) can be proved by computing both the left and right hand side explicitly.
Indeed,  the left hand side is
\begin{eqnarray*}
&&-2\int_{-1}^{1}\int_{0}^{1}\alpha(ut+1-t)^{n}t^{\alpha-1}dt\, u(1-u^{2})^{\alpha-1}du\\
&=&-\alpha\sum_{j=0}^{n}\binom{n}{j}\int_{-1}^{1}2u^{j+1}(1-u^{2})^{\alpha-1}du\int_{0}^{1}t^{j+\alpha-1}(1-t)^{n-j}dt
\\&=&-2\alpha\sum_{j=0}^{n}\binom{n}{j}\frac{\Gamma(\frac{j}{2}+1)\Gamma(\alpha)}{\Gamma(\frac{j}{2}+\alpha+1)}
\frac{\Gamma(j+\alpha)\Gamma(n-j+1)}{\Gamma(n+\alpha+1)}
\\&=&-\frac{\Gamma(\alpha+1)\Gamma(n+1)}{\Gamma(n+\alpha+1)}\sum_{j=0}^{n}\frac{(1+(-1)^{j+1})}{2}\frac{\Gamma(\frac{j}{2})\Gamma(j+\alpha)}{\Gamma(j)\Gamma(\frac{j}{2}+\alpha+1)}.
\end{eqnarray*}
The right hand side is
\begin{eqnarray*}
&&\int_{-1}^{1}\int_{0}^{1}n(u-1)(ut+1-t)^{n-1}t^{\alpha}dt(1+u)(1-u^{2})^{\alpha-1}du\\
&=&-\int_{-1}^{1}\int_{0}^{1}n(ut+1-t)^{n-1}t^{\alpha}dt(1-u^{2})^{\alpha}du\\
&=&-\sum_{j=0}^{n-1}n\binom{n-1}{j}\int_{-1}^{1}u^{j}(1-u^{2})^{\alpha}du\int_{0}^{1}t^{j+\alpha}(1-t)^{n-1-j}dt\\
&=&-\sum_{j=0}^{n-1}\frac{n!(1+(-1)^{j})}{2j!(n-1-j)!}\frac{\Gamma(\frac{j+1}{2})\Gamma(\alpha+1)}{\Gamma(\frac{j+1}{2}+\alpha+1)}\frac{\Gamma(j+\alpha+1)\Gamma(n-j)}{\Gamma(n+1+\alpha)}
\\&=&-\frac{\Gamma(\alpha+1)\Gamma(n+1)}{\Gamma(n+\alpha+1)}\sum_{j=0}^{n-1}\frac{(1+(-1)^{j})}{2}\frac{\Gamma(\frac{j+1}{2})\Gamma(j+\alpha+1)}{\Gamma(j+1)\Gamma(\frac{j+1}{2}+\alpha+1)}
\\&=&-\frac{\Gamma(\alpha+1)\Gamma(n+1)}{\Gamma(n+\alpha+1)}\sum_{j=1}^{n}\frac{(1+(-1)^{j+1})}{2}\frac{\Gamma(\frac{j}{2})\Gamma(j+\alpha)}{\Gamma(j)\Gamma(\frac{j}{2}+\alpha+1)}.
\end{eqnarray*}
Comparing both sides, we obtain the identity (\ref{q1}).

Now, with the identity (\ref{q1}), the intertwining operator for $x_{2}^{n}$ becomes
\begin{eqnarray}
\label{r1}&&V_{\kappa}(x_{2}^{n})
=\alpha\int_{-1}^{1}\int_{0}^{1} (ut+1-t)x_{2})^{n}t^{\alpha-1}dtd\mu^{\alpha}(u)
\\&&+\int_{-1}^{1}\int_{0}^{1}n(u-1)x_{2}((ut+1-t)x_{2})^{n-1}t^{\alpha}dtd\mu^{\alpha}(u).
\nonumber\end{eqnarray}
On the other hand, by direct verification or using integration by parts, we have
\begin{eqnarray}\label{e1}
\alpha\int_{0}^{1}(ut+1-t)^{n}t^{\alpha-1}dt+\int_{0}^{1}(u-1)n(ut+1-t)^{n-1}t^{\alpha}dt=u^{n}
.\end{eqnarray}
Hence, combining (\ref{r1}) and (\ref{e1}), we obtain
 \begin{eqnarray*}
V_{\kappa}(x_{2}^{n})
&=&\int_{-1}^{1}(x_{2}u)^{n}d\mu^{\alpha}(u),
\end{eqnarray*}
which is the well-known expression.
\end{ex}

\subsection{New proof of Xu's result}

In this section, we reobtain the intertwining operator given in \cite{x}. We start from the odd dihedral group $I_{k}$ with multiplicity function $\alpha$. In this case, the Laplace transform of the Dunkl kernel is given by
\begin{eqnarray}\label{dk1}&& \frac{1}{2^{k\alpha}\Gamma(k\alpha +1)}\mathcal{L}(E_{\kappa} (z, w, t))\\&=&
\int_{-1}^{1}\frac{(s+S)^{k}-2{\rm Re}(z^{k}\overline{w^{k}})+(s-S)^{k}}{(s-{\rm Re}(z\overline{w}))((s+S)^{k}-2|zw|^{k}\xi_{u,1}(k\phi_{1}, k\phi_{2})+(s-S)^{k})^{\alpha+1}}d\mu^{\alpha}(u)\nonumber\end{eqnarray}
where \[\xi_{u, v}(k\phi_{1}, k\phi_{2})=v\cos(k\phi_{1})\cos(k\phi_{2})+u\sin(k\phi_{1})\sin(k\phi_{2}).\]  This is obtained by the relations between the Dunkl kernel of $I_{k}$ and $I_{2k}$, see  also Theorem 12 in \cite{CDL}.

Denote $w_{p}=e^{i\frac{p\pi}{k}}$, then $w_{p}^{k}=e^{ip\pi}=\cos(p\pi) $, for $p=0,1,\ldots, 2k-1$. Putting  $w=w_{p}$ in  formula (\ref{dk1}), the Dunkl kernel $E_{\kappa} (z, w_{p}, t)$ in the Laplace domain becomes
\begin{eqnarray*}&& \mathcal{L}(E_{\kappa} (z, e^{i\frac{p\pi}{k}}, t))\\
&=&2^{k\alpha}\Gamma(k\alpha +1)\frac{(s+S)^{k}-2{\rm Re}(z^{k}\overline{w^{k}})+(s-S)^{k}}{(s-{\rm Re}(z\overline{w}))((s+S)^{k}-2|zw|^{k}\xi_{u,1}(k\phi_{1}, p\pi)+(s-S)^{k})^{\alpha+1}}
\\&=&\Gamma(k\alpha+1)\frac{1}{\left(s-|z|\cos(\phi_{1}-\frac{p\pi}{k})\right)^{\alpha+1}\prod_{j=1}^{k-1}\left(s-|z|\cos(\phi_{1}-\frac{p\pi}{k}-\frac{2j\pi}{k})\right)^{\alpha}}  \end{eqnarray*}
where the first identity is because \[\xi_{u, 1}(k\phi_{1}, p\pi)=\cos(k\phi_{1})\cos(p\pi)+u\sin(k\phi_{1})\sin(p\pi)=(-1)^{p}\cos(k\phi_{1})\]
which is independent of $u$ and $\int_{-1}^{1}d\mu_{\alpha}(u)=1$. The inverse Laplace transform immediately shows that the Dunkl kernel in this special case is
\begin{eqnarray} \label{x1}
&&V_{\kappa}\left(e^{\langle \cdot, w_{p}\rangle}\right)(z)=E_{\kappa} (z, w_{p})\\&=&e^{a_{0}}\Phi_{2}^{(k-1)}(\alpha, \ldots, \alpha; k\alpha+1; a_{1}-a_{0}, \ldots, a_{k-1}-a_{0})\nonumber\\&=&c_{\alpha, k}\int_{T^{k-1}}e^{\langle e^{ip\pi/k},  z \sum_{j=0}^{k-1}e^{-i2j\pi/k}t_{j}\rangle} t_{0}^{\alpha}\prod_{j=1}^{k-1}t_{j}^{\alpha-1}dt_{1}\ldots dt_{k-1}\nonumber
\end{eqnarray}
where $c_{\alpha, k}=\frac{\Gamma{(k\alpha+1)}}{\alpha\Gamma(\alpha)^{k}}$, $a_{j}=|z|\cos(\phi_{1}-\frac{p\pi}{k}-\frac{2\pi j}{k} )$, $j=0,\ldots, k-1$ and $t_{0}=1-\sum_{j=1}^{k-1}t_{j}$. This formula  also follows from  setting $w=w_{p}$ in Theorem \ref{ib111}.

It is known that the intertwining operator preserves homogenous polynomials, i.e. $V_{\kappa}(\mathcal{P}_{n})\subset\mathcal{P}_{n} $ where $\mathcal{P}_{n}$ is the space of homogenous polynomials of degree $n$.
Hence, formula (\ref{x1}) yields
\begin{eqnarray*}
&&V_{\kappa}(\langle \cdot, w_{p}\rangle^{n} )(z)\\&=&c_{\alpha, k}\int_{T^{k-1}}\left\langle w_{p}, z \sum_{j=0}^{k-1}e^{-i2j\pi/k}t_{j}\right\rangle^{n} t_{0}^{\alpha}\prod_{j=1}^{k-1}t_{j}^{\alpha-1}dt_{1}\ldots dt_{k-1}.
\end{eqnarray*}
By a limit discussion,  it  further leads to the intertwining operator for functions of the form  $f(\langle e^{i \frac{p\pi}{k}}, z\rangle  )$ as
\begin{eqnarray*}
&&V_{\kappa}\left( f\left(\langle  \cdot, e^{i \frac{p\pi}{k}}\rangle  \right)\right)(z)\\&=&c_{\alpha, k}\int_{T^{k-1}}f\left(\left\langle e^{i\frac{p\pi}{k}}, z \sum_{j=0}^{k-1}e^{-i2j\pi/k}t_{j}\right\rangle\right) t_{0}^{\alpha}\prod_{j=1}^{k-1}t_{j}^{\alpha-1}dt_{1}\ldots dt_{k-1}
\end{eqnarray*}
where $t_{0}=1-\sum_{j=1}^{k-1}t_{j}$ and $c_{\alpha, k}=\frac{\Gamma{(k\alpha+1)}}{\alpha\Gamma(\alpha)^{k}}$, which is the formula given in \cite{x}, Theorem 1.1.

Based on the above proof, we understand Xu's formula  in the another way, which is the following corollary.
\begin{corollary} \label{hl1} For  polynomials $p(z)$, the intertwining operator $V_{\kappa}$ for the dihedral group $I_{k}$ with $k$ odd at the lines $z=|z|e^{i\frac{q\pi}{k}}, 0\le q \le 2k-1$ is given by
\begin{eqnarray*}
&&V_{\kappa}\left( p\left( \cdot \right)\right)\left(|z|e^{i\frac{q\pi}{k}}\right)\\&=&c_{\alpha, k}\int_{T^{k-1}}p\left( \sum_{j=0}^{k-1}|z| e^{i(q-2j)\pi/k}t_{j}\right) t_{0}^{\alpha}\prod_{j=1}^{k-1}t_{j}^{\alpha-1}dt_{1}\ldots dt_{k-1},
\end{eqnarray*}
where $t_{0}=1-\sum_{j=1}^{k-1}t_{j}$ and $c_{\alpha, k}=\frac{\Gamma{(k\alpha+1)}}{\alpha\Gamma(\alpha)^{k}}$.
\end{corollary}
\begin{proof} The same method used for deriving the formula (\ref{x1}) leads to a similar formula for the Dunkl kernel $E_{\kappa}\left(|z|e^{i\frac{q\pi}{k}}, w\right)$. With this formula, we have
\begin{eqnarray*}
&&V_{\kappa}\left( p\left( \cdot \right)\right)\left(|z|e^{i\frac{q\pi}{k}}\right)\\
&=&\left[p(w), E_{\kappa}\left(|z|e^{i\frac{q\pi}{k}}, w\right)\right]_{0}\\&=&
c_{\alpha, k}\int_{T^{k-1}}\left[p(w), e^{\left\langle w,  |z| \sum_{j=0}^{k-1}e^{\frac{i(q-2j)\pi}{k}t_{j}}\right\rangle}\right]_{0} t_{0}^{\alpha}\prod_{j=1}^{k-1}t_{j}^{\alpha-1}dt_{1}\ldots dt_{k-1}\\
&=&c_{\alpha, k}\int_{T^{k-1}}p\left(  |z| \sum_{j=0}^{k-1}e^{\frac{i(q-2j)\pi}{k}}t_{j}\right) t_{0}^{\alpha}\prod_{j=1}^{k-1}t_{j}^{\alpha-1}dt_{1}\ldots dt_{k-1}.
\end{eqnarray*}
\end{proof}
In particular, Corollary \ref{hl1} offers a way to compute the intertwining action on polynomial $p(z)$ when $V_{\kappa}(p(\cdot))(z)$ is radial.

At the end of this section, we show that the same method leads to a partial and simple formula for the dihedral group $I_{2k}$ with $\kappa=(\alpha, \beta)$ without other difficulties. In this case, the Weyl fractional integral vanishes as well.

 \begin{ex} Consider $w_{p}=e^{i\frac{p\pi}{k}}$ and $z=|z|e^{i\frac{(q+1/2)\pi}{k}}$, $p, q=0, 1, 2, \ldots, 2k-1$. Then we have
\[\xi_{u,v}((q+1/2)\pi, p\pi)=v\cos((q+1/2)\pi)\cos(p\pi)+u\sin((q+1/2)\pi)\sin(p\pi)=0.\]In this case, the Dunkl kernel at the line $z=|z|e^{i\frac{(q+1/2)\pi}{k}}$  and $w=w_{p}$ is only given by an integral over the simplex,
\begin{eqnarray*}
&&V_{\kappa}\left(e^{\langle \cdot, w_{p}\rangle}\right)(z)=E_{\kappa} \left(z, w_{p}\right)\\&=&\frac{\Gamma(k(\alpha+\beta)+1)}{\Gamma(\alpha+\beta)^{k}}\int_{T^{k-1}}e^{\langle e^{ip\pi/k},  z \sum_{j=0}^{k-1}e^{-i2j\pi/k}t_{j}\rangle} t_{0}^{\alpha}\prod_{j=1}^{k-1}t_{j}^{\alpha-1}dt_{1}\ldots dt_{k-1}.\nonumber
\end{eqnarray*}
The same discussion as in the above shows that  for functions $f(\langle e^{i \frac{p\pi}{k}}, z\rangle  )$, the intertwining operator at the line $re^{i\frac{(q+1/2)\pi}{k}}$ is given by
\begin{eqnarray*}
&&V_{\kappa}\left( f\left(\left\langle  \cdot, e^{i \frac{p\pi}{k}}\right\rangle  \right)\right)\left(|z|e^{i\frac{(q+1/2)\pi}{k}}\right)\\&=&\frac{\Gamma(k(\alpha+\beta)+1)}{\Gamma(\alpha+\beta)^{k}}\int_{T^{k-1}}f\left(\left\langle e^{i\frac{p\pi}{k}}, |z|e^{i\frac{(q+1/2)\pi}{k}} \sum_{j=0}^{k-1}e^{-i2j\pi/k}t_{j}\right\rangle\right) \\&&\times t_{0}^{\alpha}\prod_{j=1}^{k-1}t_{j}^{\alpha-1}dt_{1}\ldots dt_{k-1},
\end{eqnarray*}
where $t_{0}=1-\sum_{j=1}^{k-1}t_{j}$.
\end{ex}
\begin{remark} The same method together with the integral expression of the generalized Bessel function leads to an integral expression of the intertwining operator for the $I_{k}$ invariant polynomials.
\end{remark}

\begin{remark} A similar approach was used to derive the intertwining operator for a special class of functions for symmetric groups, where again  the Humbert functions appear,   see \cite{dl}.
\end{remark}


\section{Conclusions}
In this paper, an integral expression of the intertwining operator and it inverse is given for arbitrary reflection groups which is based on the classical Fourier transform and the Dunkl kernel. For the dihedral case, explicit expressions for the generalized Bessel function and Dunkl kernel are obtained by inverting the Laplace domain result of our previous paper \cite{CDL} using the second class of Humbert functions. With these explicit formulas, we obtain several integral expressions for the intertwining operators in the symmetric and the non-symmetric settings.  The positivity  and the bound of the Dunkl kernel can be observed directly from our integral expressions.

\setcounter{equation}{0}


\section*{Acknowledgements}
This work was supported by the Research Foundation Flanders (FWO) under Grant EOS 30889451.


\section*{List of notations}
For the reader's convenience, we list the notations used in  Section 4 below.\\

\begin{center}
\begin{tabular}{|c||c|}
\hline
 $z$& $|z|e^{i\phi_{1}}$\\
 \hline
$w$& $|w|e^{i\phi_{2}}$\\
 \hline
$\langle z, w\rangle$ & ${\rm Re} (z\overline{w})$\\
 \hline
$d\nu^{\alpha}(u)$ &  $\displaystyle{\frac{\Gamma(\alpha+1/2)}{\sqrt{\pi}\Gamma(\alpha)}(1-u^{2})^{\alpha-1}du}$  \\
 \hline
 $d\mu^{\alpha}(u)$ & $\displaystyle{\frac{\Gamma(\alpha+1/2)}{\sqrt{\pi}\Gamma(\alpha)}(1+u)(1-u^{2})^{\alpha-1}du}$  \\
 \hline
$\xi_{u,v}(\phi_{1}, \phi_{2})$ & $v\cos(\phi_{1})\cos(\phi_{2})+u\sin(\phi_{1})\sin(\phi_{2})$ \\
 \hline
$q_{u,v}(\phi_{1}, \phi_{2})$ & $\arccos(\xi_{u,v}(\phi_{1}, \phi_{2}))$ \\
 \hline
  $a_{j}^{+} $ & $|zw|\cos\biggl(\frac{q_{u,v}(k\phi_{1}, k\phi_{2})-2j\pi}{k}\biggr)$, $0\le j\le k-1$\\
 \hline
 $a_{j}^{-}$ & $|zw|\cos\biggl(\frac{\pi-q_{u,v}(k\phi_{1}, k\phi_{2})-2j\pi}{k}\biggr), 0\le j\le k-1$\\
 \hline
 $a_{k}$ & ${\rm Re}(z \overline{w})$\\
 \hline
   $a_{j}$ & $|zw|\cos\biggl(\frac{q_{u,v}(k\phi_{1}, k\phi_{2})+2j\pi}{k}\biggr), 0\le j \le k-1$ \\
 \hline
\end{tabular}
\end{center}


\begin{thebibliography}{99}

\bibitem{b1} B. Amri, Note on Bessel functions of type $A_{N-1}$. {\em Integral Transf. Spec. Funct.} {\bf  25} (2014), 448-461.
\bibitem{AD} B. Amri, N. Demni, Laplace-type integral representations of the generalized Bessel function and of
the Dunkl kernel of type $B_{2}$. {\em Mosc. Math. J.} {\bf 17} (2017), 175-190.

\bibitem{Anker}
	J.-P. Anker, An introduction to Dunkl theory and its analytic aspects.  	{\em Analytic, Algebraic and Geometric Aspects of Differential Equations}, 3-58. Trends Math., Birkh\"auser, Chem, 2017.
	\bibitem{S2007} S. Ben~Sa{\"{\i}}d, On the integrability of a representation of $sl(2, \mathbb{R})$. {\em J. Funct. Anal.} {\bf 250} (2007), 249-264.
\bibitem{SKO}
S. Ben~Sa{\"{\i}}d, T. Kobayashi, B. {\O}rsted, Laguerre semigroup and Dunkl operators. {\em Compos. Math.}  {\bf 148} (2009), 1265-1336.

\bibitem{cw} J. F. Chamayou, J. Wesolowski, Lauricella and Humbert functions through probabilistic tools. {\em Integral Transf. Spec. Funct.} {\bf 20} (2009), 529-538.

\bibitem{CDL}  D. Constales, H. De Bie, P. Lian, Explicit formulas for the Dunkl dihedral kernel and the $(\kappa, a)$-generalized Fourier kernel. {\em J. Math. Anal.  Appl.} {\bf 460} (2018),  900-926.


\bibitem{Da} F. Dai, H. Wang, A transference theorem for the Dunkl transform and its applications.
{\em J. Funct. Anal.} {\bf 258} (2010), 4052-4074.

\bibitem{deJ}
M. de Jeu,
\newblock The Dunkl transform.
\newblock {\em Invent. Math.} {\bf 113} (1993), 147--162.

 \bibitem{DJ} M. de Jeu, Paley-Wiener theorems for the Dunkl transform. {\em Trans. Amer. Math. Soc}. {\bf 358} (2006), 4225-4250.
\bibitem{dl} H. De Bie, P. Lian, Dunkl intertwining operator for symmetric groups. arXiv:2009.02087.

\bibitem{DDY} L. Deleaval, N. Demni, H. Youssfi, Dunkl kernel associated with dihedral groups. {\em J. Math. Anal.  Appl.} {\bf 432} (2015), 928-944.

 \bibitem{DD1} L. Deleaval, N. Demni, On a Neumann-type series for modified Bessel functions of the first kind.  {\em Proc. Amer. Math. Soc}. {\bf 146} (2018),  2149-2161.
\bibitem{DD2} L. Deleaval, N. Demni, Generalized Bessel functions of dihedral-type: expression as a series of confluent Horn functions and Laplace-type integral representation. {\em Ramanujan J.} {2020}, https://doi.org/10.1007/s11139-019-00234-0.
 \bibitem{Dn} N. Demni, Generalized Bessel function associated with dihedral groups. {\em J. Lie Theory.}
{\bf 22} (2012), 81-91.



 \bibitem{D} C. F. Dunkl, Poisson and Cauchy kernels for orthogonal polynomials with dihedral symmetry,
{\em J. Math. Anal. Appl.} {\bf 143} (1989), 459-470.


\bibitem{D2} C. F. Dunkl, An intertwining operator for the group B2. {\em Glasg. Math. J.} {\bf 49} (2007), 291-319.

\bibitem{DTAMS}
	C. F. Dunkl,
	\newblock {Differential-difference operators associated to reflection groups}.
	\newblock {\em Trans. Amer. Math. Soc.} \textbf{311} (1989), 167-183.
	
\bibitem{DIA} C. F. Dunkl, Intertwining operators and polynomials associated with the symmetric group. {\em Monatsh. Math.} {\bf 126} (1998), 181-209.

\bibitem{D1} C. F. Dunkl, Intertwining operators associated to the group $S_{3}$. {\em Trans. Amer. Math. Soc.} {\bf 347} (1995), 3347-3374.

\bibitem{Ddih} C. F. Dunkl, Polynomials associated with dihedral groups.
{\em SIGMA Symmetry Integrability Geom. Methods Appl.} {\bf 3} (2007), Paper 052, 19 pp.

\bibitem{DJap} C. F. Dunkl, Reflection groups in analysis and applications. {\em Japan. J. Math.} {\bf 3} (2008), 215-246.

\bibitem{Ddunkl} C. F. Dunkl, Hankel transforms associated to finite reflection groups. In: Proceedings of the special session on hypergeometric functions on domains of positivity, Jack polynomials and applications at AMS meeting in Tampa, Fa March 22 23. (Contemp. Math. 138 (1992)) Providence, RI: Am. Math. Soc.

\bibitem{DV}  C. F. Dunkl, Integral kernels with reflection group invariance.
{\em Canad. J. Math.} {\bf 43} (1991), 1213-1227.

\bibitem{DdO}
C. F. Dunkl, M. de Jeu, E. Opdam,
Singular polynomials for finite reflection groups.
{\em Trans. Amer. Math. Soc.} {\bf 346} (1994), 237-256.


\bibitem{DX} C. F. Dunkl, Y. Xu, {\em Orthogonal polynomials of several variables (second edition). } Cambridge university press, 2014.

\bibitem{Dz} J. Dziubanski, A. Hejna, H\"ormander's multiplier theorem for the Dunkl transform. {\em J. Funct. Anal.} {\bf 277} (2019), 2133-2159.


\bibitem{E2} A. Erd\'{e}lyi ed, {\em Tables of integral transforms.} Vol. 1. New York: McGraw-Hill, 1954.

\bibitem{EH} H. Exton, {\em Multiple Hypergeometric Functions and Applications. }Ellis Horwood, 1983.

\bibitem{GB2} L. C. Grove, C. T. Benson, {\em  Finite Reflection Groups,} 2nd edn.  Springer, Berlin-New York, 1985.

\bibitem{H} P. Humbert, The confluent hypergeometric functions of two variables. {\em Proc. Roy. Soc. Edinburgh.} {\bf 41} (1920), 73-82.
 \bibitem{M2} I. G.  Macdonald,  The Volume of a Compact Lie Group. {\em Invent. Math.} {\bf 56} (1980), 93-95.

\bibitem{rm} M. R\"{o}sler, {\em  Dunkl operators: theory and applications, Orthogonal polynomials and special functions. } Springer Berlin Heidelberg, 2003.





\bibitem{rm1} M. R\"{o}sler,  Positivity of Dunkl's intertwining operator. {\em Duke Math. J.}  {\bf 98} (1999), 445-464.

\bibitem{rdj} M. R\"{o}sler, M. de Jeu, Asymptotic analysis for the Dunkl kernel. {\em J. Approx. Theory } {\bf 119}(2002), 110-126.
\bibitem{Ps} P. Sawyer, A Laplace-type representation of the generalized spherical functions associated to the root systems of type A. {\em Mediterr. J. Math.} {\bf 14}(2017), 147.

\bibitem{x} Y. Xu, Intertwining operators associated to dihedral groups. {\em Constr. Approx.} 2019. https://doi.org/10.1007/s00365-019-09487-w.


\end{thebibliography}
\end{document}